\documentclass[11pt]{amsart}
\usepackage[T1]{fontenc}
\usepackage[utf8]{inputenc}
\usepackage[english]{babel}
\usepackage{amsmath, amsthm, amssymb, amsfonts, enumerate, mathrsfs, dsfont, comment, bm, mathtools}
\usepackage{algpseudocode}
\usepackage{enumitem}
\usepackage[colorlinks=true,linkcolor=blue,urlcolor=blue]{hyperref}
\usepackage{color, geometry, todonotes, indentfirst, graphicx, pdfpages}
\usepackage[title]{appendix}
\usepackage{fancyhdr}
\usepackage{polski}
\usepackage[dvipsnames]{xcolor}
\allowdisplaybreaks
\newcommand{\jo}[1]{{\textcolor{blue}{#1}}}

\newtheorem{theorem}{Theorem}[section]
\newtheorem{remark}[theorem]{Remark}
\newtheorem{assumption}[theorem]{Assumption}
\newtheorem{lemma}[theorem]{Lemma}

\newtheorem{definition}[theorem]{Definition}
\newtheorem{example}[theorem]{Example}
\theoremstyle{plain}

\def \R{\mathbb{R}}

\def \P{\mathbb{P}}
\def \E{{\mathbb{E}}}
\def \F{\mathbb{F}}
\def \cP{\mathcal{P}}
\def \cF{{\mathcal F}}

\def \T{{\mathcal T}}

\def \A{\mathcal{A}}

\def \and{\quad \text{and} \quad}

\DeclareMathOperator*{\argmax}{arg\,max}

\usepackage[ruled,vlined, linesnumbered]{algorithm2e}
\usepackage{dsfont}
\usepackage{wasysym}

\usepackage{accents}

\usepackage{amsmath,amsfonts,amsthm}
\usepackage{color}
\usepackage{hyperref}
\theoremstyle{definition}

\usepackage{amssymb,fullpage}
\usepackage{enumerate}
\usepackage{graphicx}
\usepackage{subfig}
\usepackage[utf8]{inputenc}
\usepackage{wrapfig}
\usepackage{mathrsfs}
\usepackage{latexsym}
\usepackage{amsbsy}
\usepackage{color}
\usepackage{bbm}
\usepackage{mathrsfs}
\usepackage{todonotes}
\usepackage{url}

\usepackage{wasysym}

\providecommand{\otherindexspace}[1]{}

\definecolor{falured}{rgb}{0.5, 0.09, 0.09}
\hypersetup{colorlinks=true, linkcolor={blue}, citecolor={falured}}

\usepackage{comment}

\usepackage{dsfont}

\title[Entropy Regularization for MFGs of Optimal Stopping]{Entropy Regularization in Mean-Field Games of Optimal Stopping}

\author[J. Dianetti]{Jodi Dianetti}
\author[R. Dumitrescu]{Roxana Dumitrescu}
\author[G. Ferrari]{Giorgio Ferrari}
\author[R. Xu]{Renyuan Xu}

\thanks{J. Dianetti: Department of Economics and Finance, University of Rome Tor Vergata. Email: \href{mailto:jodi.dianetti@uniroma2.it}{jodi.dianetti@uniroma2.it}}
\thanks{R. Dumitrescu: ENSAE-CREST, Institut Polytechnique de Paris. Email: \href{mailto:roxana.dumitrescu@ensae.fr}{roxana.dumitrescu@ensae.fr}}
\thanks{G. Ferrari: Center for Mathematical Economics, Bielefeld University. Email: \href{mailto:giorgio.ferrari@uni-bielefeld.de}{giorgio.ferrari@uni-bielefeld.de}}
\thanks{R. Xu: Management Science \& Engineering, Stanford University. Email: \href{mailto:renyuanxu@stanford.edu}{renyuanxu@stanford.edu}}

\date{\today}

\numberwithin{equation}{section}

\begin{document}

\begin{abstract}
We study mean-field games of optimal stopping (OS-MFGs) and introduce an entropy-regularized framework to enable learning-based solution methods. By utilizing randomized stopping times, we reformulate the OS-MFG as a mean-field game of singular stochastic controls (SC-MFG) with entropy regularization. We establish the existence of equilibria and prove their stability as the entropy parameter vanishes. Fictitious play algorithms tailored for the regularized setting are introduced, and we show their convergence under both Lasry–Lions monotonicity and supermodular assumptions on the reward functional. Our work lays the theoretical foundation for model-free learning approaches to OS-MFGs.
\end{abstract} 
\maketitle  
\smallskip 
{\textbf{Keywords}}:
Mean-field game of optimal stopping, singular stochastic control, entropy regularization, randomized stopping times, fictitious play algorithm.

\smallskip 
{\textbf{AMS subject classification}}: 
91A16, 
60G40,  
93E20, 
68T01  


\section{Introduction}
\label{sec:intro} 

The study of mean-field games (MFGs) has become central to the analysis of large-population stochastic control systems, where individual agents interact through the empirical distribution of states and/or controls. For methodologies, techniques, and applications, we refer to the two-volume book \cite{CarmonaDelarue18}.
In this paper, we focus on a specific class of such problems: mean-field games of optimal stopping (OS-MFGs). Optimal stopping has a wide range of applications in Economics, Finance, Operations Research, and other applied fields. Relevant examples include pricing American derivatives, entry-exit problems (real options models), and optimal timing for buying or selling an asset, among others.
On the other hand, OS-MFGs  have received attention only very recently. Among the fundamental contributions on this topic (without learning considered), we first mention  \cite{nutz2018mean}, where the authors consider a specific game in which the interaction occurs through the number of players who have already stopped. In  \cite{carmona.delarue.lacker.2017.timing}, the authors study an optimal stopping game with Brownian common noise, inspired by a model of bank runs. They prove the strong existence of mean-field equilibria in a setting with strategic complementarity, and weak existence in a setting with continuity. In \cite{Bertucci}, a purely analytical approach is adopted to solve the problem -- namely, the study of the system of coupled Hamilton-Jacobi-Bellman and Fokker-Planck equations -- and the existence of mixed solutions is proved. { More recently, a compactification approach based on the linear programming formulation is developed in a series of papers \cite{bouveret.dumitrescu.tankov.20}, \cite{dumitrescu2021control} and \cite{dumitrescu.leutscher.tankov.2022linear}, in which the existence of relaxed solutions, interpreted via occupational measures, is established. These techniques have been then applied to solve entry-exit optimal stopping games in electricity markets in \cite{aidentry} and \cite{bassiere2024mean}, and in the case with common noise in \cite{dumitrescu2024energy}. In \cite{roxanaetal}, a rigorous connection between the occupation measures arising in the linear programming approach and randomized stopping is established.} In addition, \cite{possamai2023mean} studies an optimal stopping game through the lens of the master equation. Another recent contribution is \cite{he2023mean}, where the authors develop a new approach to solving various types of MFGs {(including optimal stopping)} based on a mean-field version of the Bank-El Karoui representation theorem for stochastic processes.  In the context of optimal stopping MFGs with partial information, \cite{campbell2024mean,campbell2025bayesian,souganidis2025mean} study the existence of the mean-field solution for a one-dimensional filtering problem under different model set-ups. Finally, related works on {McKean-Vlasov optimal stopping} are \cite{talbi2023dynamic} and \cite{cosso2025mean}.

Despite their importance, OS-MFGs pose significant theoretical and computational challenges. First,  they require more subtle techniques to be solved than MFGs with standard control because of the irregularity of the flow measure generated by the simultaneous exit of a significant number of players. This class of games becomes particularly challenging when the goal is to develop model-free, learning-based solution methods. 

The primary motivation of our work is to establish the foundations of a suitable OS-MFG framework that supports model-free solution methods. A central challenge in this direction is to rigorously formalize and analyze randomized and explorative strategies in games, which play a crucial role in the design of reinforcement learning (RL) algorithms. Establishing the existence and uniqueness of (mean-field) equilibria that accommodate such strategies is therefore essential—not only for theoretical completeness but also for enabling practical learning algorithms with provable guarantees. In the absence of such a foundation, rigorous analysis of multi-agent learning remains elusive. Moreover, the convergence of RL methods often depends on the properties of the underlying iterative schemes (e.g., fictitious play), whose well-posedness similarly hinges on the existence and uniqueness of equilibria in games that allow for randomized behavior (see, e.g., \cite{guo2019learning,guo2023general}). This need becomes even more pronounced in environments characterized by irregular or sparse decision-making, such as optimal stopping problems, where agents make a single, irreversible decision based on limited information. Unlike standard Markov Decision Processes (MDPs) \cite{puterman2014markov,sutton1998reinforcement}, where frequent actions and rewards can reduce the reliance on explicit exploration, optimal stopping games require more deliberate and carefully designed exploration mechanisms to facilitate information acquisition \cite{devidze2022exploration}. On the technical side, however, a major challenge arises in applying RL to OS-MFGs. Standard equilibrium-finding algorithms in MFGs typically rely on the stability of optimal controls with respect to variations in the distribution of agents. This stability, however, is particularly difficult to ensure in the context of optimal stopping, where understanding how equilibrium stopping policies respond to small perturbations in the measure flow is highly nontrivial. To illustrate this, consider that in a Markovian setting, equilibrium stopping rules are often given by hitting times of the underlying state process at a free boundary. Determining the stability properties (e.g., Lipschitz continuity) of these free boundaries is already a technically intricate problem, even in single-agent optimal stopping settings (see, e.g., \cite{de2019lipschitz} and the discussion therein). 
\vspace{0.15cm}

\textbf{Our contributions.} To address these challenges, we develop a new theoretical framework that enables the use of model-free algorithms for OS-MFGs. The core idea is to reformulate the game as a MFG of singular stochastic controls (SC-MFG) by introducing randomized stopping times (see \cite{dianetti2024exploratory, touzi.vieille.2002continuous}). A randomized stopping time can be interpreted as the (conditional) probability of stopping before a certain time \( t \). This interpretation leads naturally to an adapted, nonnegative, nondecreasing process with right-continuous sample paths, bounded by 1 -- an object that, following terminology from control theory, we refer to as a singular control (see Chapter VIII in \cite{fleming.soner2006} for an introduction to singular stochastic controls). 
To promote exploration, we introduce an entropy regularization term into the stopping functional, as previously established in the single-agent setting of \cite{dianetti2024exploratory}. This regularization term is governed by a temperature parameter $\lambda \geq 0$ and induces strict concavity in the performance criterion. This change enables optimizers to deviate from pure $0\text{-}1$ strategies (of not stopping or stopping) and instead favor reflecting-type controls.  {In the context of MFG, the equilibrium is characterized via a two-step procedure.} First, given a fixed flow of measures and a joint distribution, we determine the singular control that maximizes the regularized reward functional. Second, we impose the so-called consistency condition, which requires that the fixed input measures coincide with the distribution of the state process before stopping and the joint distribution of the (randomized) stopping time and the stopped state. Naturally, in this framework, the consistency conditions must be reformulated in terms of randomized stopping times -- namely, in terms of the random Borel measure on the time axis induced by the singular control. The particular structure of these consistency conditions prevents us from applying existing existence results for mean-field equilibria in singular control games (see, e.g., \cite{Campietal, denkert.horst.2023extended, fu2023extended, Fu&Horst17}). We therefore establish a novel existence result for equilibria in our SC-MFG framework, which we believe is of independent theoretical interest (see Theorem \ref{theorem existence MFGE for SC-MFG}).

Having proved the existence of equilibria in randomized stopping times, we then move on to establishing $\lambda$-stability as the temperature parameter $\lambda \downarrow 0$. We show that any equilibrium of the SC-MFG approximates, as $\lambda \downarrow 0$ (up to a subsequence),  an equilibrium of the original MFG of optimal stopping without entropy regularization. Furthermore, under a Lasry-Lions monotonicity condition, we can prove a stronger result: the entire sequence of mean-field equilibria of the singular control game (parametrized by $\lambda$) converges to an equilibrium of the original optimal stopping game {(see Theorem \ref{theorem qualitative stability in lambda})}. This result ensures that the solutions of the regularized problems converge to equilibria of the original OS-MFG, providing a crucial theoretical justification for using entropy-based approximations in RL methods.

We then design novel fictitious-play algorithms tailored to the entropy regularized SC-MFG framework. We establish their convergence under both the Lasry–Lions monotonicity condition (see Theorem \ref{theorem fictitious monotone}) and in the case of supermodular reward structures (see Theorem \ref{theorem fictitious supermodular}). 
It is important to note that the presence of the entropy regularization term introduces nonlinearity into the performance criterion, placing our model outside the scope of recent approaches based on linear programming for  MFGs \cite{bouveret.dumitrescu.tankov.20, dumitrescu2021control, dumitrescu.leutscher.tankov.2022linear}. This nonlinearity necessitates the development of new analytical and numerical tools.
The fictitious-play algorithm under the Lasry–Lions condition therefore extends previous algorithms developed under this monotonicity assumption to cases where the payoff functional is non-linear with respect to the control. 
Under the supermodularity condition, the fictitious-play algorithm is new in the literature.
The theoretical foundation presented here will be complemented by a companion paper focused on the algorithmic design, related theoretical convergence analysis, and numerical validation of the proposed framework, which is currently in progress \cite{jodietal.}.
\vspace{0.15cm}

\textbf{Related literature.} Reinforcement learning applied to optimal stopping problems is closely tied to the broader challenge of learning in environments with sparse rewards \cite{devidze2022exploration,hare2019dealing}. In these settings, a reward is only granted at the stopping time, resulting in extreme sparsity and significant learning challenges compared to more conventional control tasks, whether formulated in continuous time or  classical MDPs.

When the model is fully known, recent studies such as \cite{reppen2022neural} and \cite{soner2023stopping} have introduced deep learning techniques to approximate optimal stopping boundaries. In a related vein, \cite{bayer2021randomized} investigated randomized formulations of optimal stopping and provided convergence results for both forward and backward Monte Carlo-based optimization schemes. Similarly, \cite{ata2023singular} showcased the effectiveness of deep learning methods for solving high-dimensional singular control problems.

In the realm of continuous-time RL, there is an active line of work that has been focused on regular controls, see for example \cite{bo2025optimal,jia2022policy,jia2022policy-m,ma2024convergence,wang2018exploration} to name a few. In addition, a notable advancement comes from \cite{denkert2024control}, who proposed a general policy gradient framework applicable to a wide array of stochastic control problems, including optimal stopping, impulse control, and switching. Their approach leverages connections between stochastic control and its randomized counterparts. However, their framework currently lacks theoretical guarantees for convergence. The work \cite{dong2024randomized} studies a regularized version of the one-dimensional American Put option under a Shannon entropy framework (see also \cite{Haoyang-etal} for a similar approach to impulse control problems). By employing an intensity-based control formulation, they introduce exploration and prove convergence of the policy iteration algorithm (PIA) for fixed temperature parameters. Nevertheless, the convergence of the optimal policy as the temperature parameter $\lambda \downarrow  0$ remains an open question. An alternative perspective was recently offered by \cite{dai2024learning}, who reformulated the optimal stopping problem as a regular control problem with binary ($0$ or $1$) actions and applied entropy regularization. This transformed the problem into a classical entropy-regularized control setting, enabling the use of standard RL algorithms. The very recent paper \cite{liang2025reinforcement} proposes a continuous-time reinforcement learning framework for a class of singular stochastic control problems without entropy regularization and generalizes the existing policy evaluation theories with regular controls to learn  optimal singular control law and develop a policy improvement theorem.

Another important research direction leverages non-parametric statistical methods for stochastic processes to design learning-based control algorithms \cite{christensen2023data2,christensen2023nonparametric,christensen2024learning,christensen2023data}. In particular, \cite{christensen2023nonparametric} and \cite{christensen2024learning} addressed one-dimensional singular and impulse control problems by learning critical thresholds using a non-parametric diffusion framework. This approach was extended to multivariate reflection problems in \cite{christensen2023data}, while \cite{christensen2023data2} contributed theoretical guarantees, including regret bounds and non-asymptotic PAC estimates.

Entropy regularized MFGs and their reinforcement learning (RL) formulations have received increasing attention over the past year. We do not attempt to provide an exhaustive literature review here, but instead highlight a few key references: \cite{carmona2023model}, which presents an RL approach to model-free MFGs for MDPs; \cite{anahtarci2023q}, which discusses Q-learning in regularized MFGs for MDPs; \cite{carmona2021convergence}, which provides a convergence analysis of machine learning algorithms for mean-field control and games over a finite time horizon in continuous time and space; and \cite{firoozi.jaimungal.2022exploratory, guo2022entropy}, which addresses entropy regularized MFGs in diffusion settings.

As far as it concerns the more specific class of entropy-regularized MFGs of optimal stopping -- as those considered in this paper --  the only paper brought to our attention is \cite{yu2025major}. This paper investigates a discrete-time major-minor MFG where the major player can choose to either control or stop. The goal is to find a relaxed (randomized) stopping equilibrium, formulated as a fixed point of a set-valued map, whose direct analysis is difficult due to the major player's influence. To address this, the authors introduce entropy regularization for the major player's problem and reformulate the minor players' stopping problems using linear programming over occupation measures, in the spirit of \cite{bouveret.dumitrescu.tankov.20}. They prove the existence of regularized equilibria via a fixed-point theorem and show that, as regularization vanishes, these equilibria converge to a solution of the original problem, thus establishing the existence of a relaxed equilibrium. Although there are clear connections between \cite{yu2025major} and the present paper, the mathematical analysis is entirely different, as the problem in \cite{yu2025major} involves a Stackelberg game (which is not present in our case) on top of the mean-field interaction, within a discrete-time and discrete-space setting -- unlike our continuous-time and continuous-space framework.

\vspace{0.15cm}

\textbf{Organization of the paper.} The rest of the paper is organized as follows. In Section \ref{sec:preliminaries} we present the OS-MFG problem and introduce its entropy-regularized formulation through randomized stopping times. In Section \ref{sec:existence} we prove existence and stability of equilibria, while in Section \ref{sec:fictplay} we present the fictitious play algorithms and their convergence. Finally, Appendix \ref{sec:App} collects results on a connection between the studied MFG of singular control and (another) auxiliary MFG of optimal stopping.


\subsection{General notation} 
Let $p\geq 1$ and $(E,d)$  a non-empty Polish space. Define the set
$\mathcal P _p^{sub} (E)$ of subprobability measures $\nu$ on the Borel subsets of $E$ such that $\int_E d(z,z_0)^p \nu(dz) < \infty$, for some (and thus all) $z_0 \in E$.
The space $\mathcal P _p (E)$ represents the set of elements of $\mathcal P _p^{sub} (E)$ which are probability measures, and it is 
endowed with the Wasserstein distance $d_p$ defined as
\begin{align}\label{probdist}
d_p(\bar \nu, \nu) := \inf_{\pi \in \Pi(\bar \nu, \nu) } \int_{E \times E} d(\bar z,z)^p \pi (d\bar z,d z),
\end{align}
where $ \Pi(\bar \nu, \nu)$ is the set of $\pi \in  \mathcal P _p (E \times E)$ which has $\bar \nu$ as first marginal and $\nu$ as second marginal.
Following \cite[Section B]{claisse2023mean}, to define the Wasserstein distance $d^\prime_p$ on $\mathcal{P}^{sub}_p(E)$, we introduce a cemetery point $\partial$ and obtain the enlarged space $\bar{E}:=E \cup \partial$. By defining $d(z,\partial):=d(z,z_0)+1$, we extend the definition of $d$ on $\bar{E}$, in such a way that $(\bar{E},d)$ is still Polish. We define the classical Wasserstein distance $d_p$ on $ \bar{E}$ using the same expression $\eqref{probdist}$. The Wasserstein type distance $d_p^\prime$ on $\mathcal{P}^{sub}_p(\bar{E})$ is given by
\begin{align*}
d_p^{\prime}(\mu,\nu):=d_p(\bar{\mu},\bar{\nu}),
\end{align*}
where 
$$\bar{\mu}(\cdot):=\mu(\cdot \cap E)+(1-\mu(E))\delta_{\partial},\,\,\ \bar{\nu}(\cdot):=\nu(\cdot \cap E)+(1-\nu(E))\delta_{\partial}. $$
For $p=1$, we recall the following duality result (see Lemma B.1. in \cite{claisse2023mean}): for all $\mu,\nu \in \mathcal{P}_1^{sub}(E)$,
\begin{align}
d^{\prime}_1(\mu,\nu)=\sup_{\text{Lip}_1^0(E)}\{\mu(\varphi)-\nu(\varphi) \}+|\mu(E)-\nu(E)|,
\end{align}
where $\text{Lip}_1^0(E)$ denotes the collection of all functions $\varphi:E \to \mathbb{R}$ with Lipschitz constant smaller or equal to $1$ and such that $\varphi(z_0)=0$.\\
Fix $T>0$. We introduce the set $M_p(E)$ of measurable functions $m:[0,T]\to \mathcal P _p^{sub} (E)$ identified $dt$-a.e.\ on $[0.T]$, endowed with the convergence in measure topology which is induced by the metric
\begin{align}
d_p^M(m,\nu):=\int_0^T 1 \wedge d_p^\prime(m_t,\nu_t)dt,\,\, m,\nu \in M_p(E).
\end{align}
{A sequence $(m^n)_n$ converges to $m$ ($m^n \to m$, in short) if for any $\varepsilon>0$ we have
$$
\int_0^T \mathds 1 _{\{ t \, : \, d^\prime_p (m^n_t,m_t) >\varepsilon \}} dt \to 0.
$$
Recall that, for any convergent sequence $m^n \to m$, there exists a subsequence $(m^{n_k})_k$ such that
\begin{equation}
\label{eq conv in dt after conv in measure}
d^\prime_p (m^n_t,m_t) \to 0, \quad dt\text{-a.e.}
\end{equation}
We refer to Appendix B in \cite{dumitrescu.leutscher.tankov.2022linear} for further details.}

\noindent Throughout the paper, $C>0$ will denote a generic constant that may change from line to line.





\section{MFGs of optimal stopping and entropy regularization}
\label{sec:preliminaries}

This section formally introduces the OS-MFG problem with randomization and entropy regularization, leading to a singular control formulation. More specifically, Section~\ref{sec:preliminary_assumption} introduces preliminaries and model assumptions; Section~\ref{sec:optimal_stopping} describes the OS-MFG problem; and Section~\ref{secreg} develops the framework for randomized stopping times, entropy regularization, and the resulting singular control problem.

\subsection{Preliminaries and assumptions}
\label{sec:preliminary_assumption}
Let $T \in (0,\infty)$ be given and fixed throughout the rest of the paper. For $d,d_1 \in \mathbb N \setminus \{0\}$, consider the continuous functions
$ b: [0,T] \times \mathbb{R}^{d} \rightarrow \mathbb{R}^{d}, \ \sigma:[0,T] \times \mathbb{R}^{d} \rightarrow \mathbb{R}^{d \times d_1}, \  f :[0,T] \times \mathbb R ^d \times \cP^{sub} (\mathbb R ^d) \to \mathbb R$, $g : \mathbb R ^d \times \cP (\mathbb R ^d) \to \mathbb R$.



Define 
$$
M_p:=M_p(\mathbb{R}^d) \text{ and }\mathcal M : = M_p \times \mathcal P _p ( [0,T] \times  \R^d ),
$$ 
endowed with the product topology.
Elements of $\mathcal M$ will typically be denoted by couples $(m,\mu)$.
Moreover, given a sequence $(m^n,\mu^n)_n \subset \mathcal M$ and a limit point $(m,\mu) \subset \mathcal M$, we will write
$$
(m^n,\mu^n)_n \to (m,\mu)
$$
to denote the convergence in the product topology.

On a complete probability space $(\Omega, \mathcal F, \mathbb P )$, consider a $d_1$-dimensional Brownian motion $W:=(W_t)_t$ and a $\R^d$-valued square integrable $\mathcal{F}_0$-measurable random variable $x_0$ independent of $W$.  
Denote by $\mathbb F ^{x_0,W}:= (\mathcal F ^{x_0,W}_t)_{t}$  the right-continuous extension of the filtration generated by $W$ and $x_0$.
For a time-horizon $T \in (0,\infty)$, let $\mathcal T$ be the set of $\mathbb F ^{x_0,W}$-stopping times $\tau$ such that $\tau\leq T$ a.s..  

\begin{assumption}
    \label{assumption}
    The data of the problem $(x_0,b,\sigma,f,g)$ verify: 
    \begin{enumerate}
    \item\label{assumption: integrability x0} $\E [|x_0|^p] < \infty $. 
    \item\label{assumption: condition b sigma} There exists a constant $L>0$ such that
    $$ \begin{aligned}
         & | b(t, x)| + | \sigma (t, x)|  \leq L (1+|x|),\quad \forall t \in [0,T], \ x \in \mathbb R ^d , \\
        &|b (t, \bar x) - b (t,x)| + |\sigma (t, \bar x) - \sigma (t,x)|  \leq L |\bar x - x|, \quad  \forall t \in [0,T], \ x, \bar x \in \mathbb R ^d.
    \end{aligned}
    $$ 
    \item\label{assumption: condition f g} There exists a constant $K>0$ such that
    $$
    \begin{aligned}
    |f(t,x,m)| &\leq K \Big(1+|x|^p + \int_{\R^d} |z|^p m(dz) \Big), \quad \forall (t,x,m) \in [0,T] \times \mathbb R ^d \times \cP_p^{sub} (\R^d),\\
    |g(x,\mu)| &\leq K \Big(1+|x|^p + \int_{[0,T] \times \R^d} |z|^p \mu(ds,dz) \Big), \quad \forall (x,\mu) \in  \mathbb R ^d \times \cP_p ([0,T] \times \R^d).
    \end{aligned}
    $$
    \item The functions $f$ and $g$ are continuous, and $g$ is continuous in $\mu$, locally uniformly with respect to $x$. That is, there exist a constant $K>0$ and a function $w_g : \cP_p ([0,T] \times \R^d) \times \cP_p ([0,T] \times \R^d) \to [0,\infty)$  with $\lim_{d_p(\mu, \bar \mu) \to 0} w_g(\bar \mu, \mu) = 0$ such that  $|g(x,\bar \mu) - g(x, \mu) | \leq K (1 + |x|^p) w_g(\bar \mu, \mu)$ for any $x \in  \mathbb R ^d, \mu, \bar \mu \in \cP_p ([0,T] \times \R^d)$.
    
   \end{enumerate}
\end{assumption}

\subsection{The MFG of optimal stopping}\label{sec:optimal_stopping}

Consider the mean-field game of optimal stopping (OS-MFG, in short) in which, for a given behavior  $(m,\mu) \in \mathcal M$ of the population of players, the representative player maximizes, over the choice of stopping times $\tau \in \T$, the profit functional
\begin{equation}
\label{eq}
\begin{aligned}
     J(\tau , m,\mu) &:= \mathbb{E}\left[\int_{0}^{\tau} f\left(t, X_{t}, m_{t}\right) d t + g\left(X_{\tau}, \mu\right)\right], \\
\text{subject to } dX_t &= b(t, X_t )dt + \sigma (t,X_t) dW_t, \ t \in [0,T], \  X_0 =x_0.
\end{aligned}
\end{equation}
Thanks to Assumption \ref{assumption}, there exists a unique strong solution to the stochastic differential equation (SDE) for $X$ with continuous paths a.s., and such that
\begin{equation}
    \label{eq estimate SDE}
    \E \Big[ \sup_{t \in [0,T] } |X_t|^p \Big] < \infty.
\end{equation}
Moreover, the profit functional $J(\tau, m,\mu)$ is well defined for any $\tau \in \mathcal T$, the maximization problem has finite value (i.e., $\sup_{\tau } J(\tau , m,\mu) < \infty$), and  there exists an  optimal stopping (i.e., $\tau^* \in \argmax_\tau J(\tau, m, \mu)$). We refer the interested  reader to Appendix D in \cite{karatzas.shreve.98.finance} for further details.

We consider the following notion of equilibrium.  
\begin{definition}[OS-MFG equilibrium with strict stopping]\label{def:equilibrium}
An \textit{optimal stopping MFG equilibrium (with strict stopping)} is a triple $(\hat \tau, \hat m, \hat \mu)  \in  \mathcal{T} \times \mathcal{M}$ such that
\begin{equation}\label{eq def equilibrium OS MFG}
\begin{cases}
& \begin{matrix} \hat \tau \in \argmax_{\tau \in \T}J(\tau , \hat m,\hat \mu),\end{matrix}  \\ 
& \hat m_{t}(A)=\mathbb{P}\left( X_{t} \in A, t\leq \hat \tau\right), \quad \forall A \in \mathcal{B}\left(\R^d\right), \  t \in[0, T),\\
& \hat \mu\left(B \times [0,t] \right)=\mathbb{P}\left(X_{\hat \tau} \in B, \hat \tau \leq t\right), \quad \forall B \in \mathcal{B}\left(\R^d\right), \ t \in [0, T]. 
\end{cases}
\end{equation}
The first condition is the optimality condition, while the second and the third  condition correspond to the so-called consistency conditions.
\end{definition}

 Under the equilibrium notion in Definition \ref{def:equilibrium}, Problem \ref{eq} introduces a MFG where agents interact through the running reward function $f$ with instantaneous mean-field measure $m_t$ and through the terminal cost $g$ with measure $\mu$ quantifying the stopping time and stopping position.

\subsection{Singular control formulation and entropy regularization}
\label{secreg}
In this subsection, we reformulate the OS-MFG problem with the aim of providing a theoretical framework that enables agents to learn equilibria using RL algorithms.  To this end, we modify the original OS-MFG framework by incorporating the exploration–exploitation trade-off, which is empirically known to enhance the convergence of learning methods. Fictitious play algorithms will be discussed later in Section~\ref{sec:fictplay}.


\subsubsection{Randomized stopping times and singular controls}

We first introduce the concept of randomized stopping times. We borrow ideas from the literature in game theory (see, e.g., \cite{touzi.vieille.2002continuous}) and from \cite{dianetti2024exploratory}, where the entropy regularized version of an optimal stopping problem is considered.

A randomized stopping time can be intrinsically connected to singular control when interpreted as the conditional probability of stopping before a given time 
$t$. To build this connection formally, define the set of stochastic processes 
\begin{align*}
    \A:= \big\{ \xi : & \, \Omega \times[0, T] \to [0,1], \text{ $\F^{x_0,W}$-adapted, nondecreasing, c\`adl\`ag, with $\xi_{0-}=0$ and $\xi_T =1$} \big\}. 
\end{align*}
In the rest of this paper, with reference to the terminology of stochastic control theory, we shall refer to an element of $\A$ as to a singular control (see, e.g., Chapter VIII in \cite{fleming.soner2006}).

Consider a random variable $U:\Omega \to [0,1]$ which is uniformly distributed and independent from $W$ and $x_0$.
Unless to enlarge the original probability space, we can assume $U$ to be defined on $(\Omega, \mathcal F, \P)$ as well, and to be $\mathcal{F}$-measurable.
Given $\xi \in \mathcal A$, a randomized stopping time is defined as 
$$
\tau^\xi := \inf \left\{t \in [0,T] \, | \, \xi_{t}> U \right\},
$$
with the convention $\inf \emptyset = T$.

Notice that $\tau^\xi$ is a stopping time with respect to the enlarged filtration $\mathbb F ^{x_0,W,U}$, generated by $x_0, W$ and $U$. Hence $\tau^\xi$ is not necessarily an $\mathbb F ^{x_0,W}$-stopping time.
However, if $\tau \in \mathcal T$, the process $\xi^\tau \in \mathcal A$ defined by $\xi^\tau := ( \mathds 1 _{\{ t \geq \tau \} } )_t$ is such that $\tau = \tau^{\xi^\tau}$. 
This gives a natural inclusion of $\left\{(\xi^\tau_t)_t,\,\, \xi_t^\tau=\mathds 1 _{\{ t \geq \tau  \}} \right\}$ into $\mathcal A$.
We finally observe that, thanks to the definition of $\tau^\xi$ and the independence of $U$ with respect to $W$ and $x_0$, we have
\begin{equation}\label{eq key relation randomized ST}
\mathbb{P} \big( \tau^\xi \leq t \mid \cF _{t}^{x_0,W}\big)
=\mathbb{P}\left( U \leq {\xi}_{t} \mid \cF _{t}^{x_0,W}\right)
=\int_{0}^{\xi_{t}} d u=\xi_{t}, 
\end{equation}
so that $\xi_t$ can be interpreted as the (conditional) probability of stopping before time $t$.

With slight abuse of notation, we will evaluate the payoff $J$ as in \eqref{eq} on randomized stopping times as well.
In particular, using \eqref{eq key relation randomized ST}, an application of the tower property and of the independence of $U$ with respect to $(W,x_0)$ allows to rewrite the payoff functional for any $(m,\mu) \in \mathcal M$ as
\begin{equation}
\label{eq functional J for tau xi}
J(\tau^\xi,m,\mu) = \E  \left[  \int_{0}^{T} f\left(s, X_{s}, m_{s}\right)\left(1-\xi_{s}\right) d s+\int_{[0,T]} g\left( X_{s,}, \mu\right) d \xi_{s} \right].
\end{equation}
Here, $\xi \in \mathcal A$ is the singular control associated to the randomized stopping time $\tau^{\xi}$, and for a generic continuous process $Y:=(Y_t)_t$ we have set
$
\int_{[0,T]} Y_s d \xi_{s} := Y_0 \xi_0 + \int_0^T Y_s d\xi_s.
$


\subsubsection{Consistency conditions for randomized strategies}

Equation \eqref{eq functional J for tau xi} provides the expression of the representative player's payoff functional when she is allowed to play a randomized stopping rule. We now derive a convenient form of the consistency conditions in \eqref{eq def equilibrium OS MFG} for randomized stopping times $\tau^\xi$ in terms of the related singular control $\xi \in \A$.

For any $\xi \in \A$, by using \eqref{eq key relation randomized ST}, we rewrite the first consistency condition in \eqref{eq def equilibrium OS MFG} as
\begin{equation}\label{eq consistency m singular control}
\begin{aligned}
\mathbb{P}\left(X_{t} \in A, t < \tau ^\xi \right) 
& =\mathbb{P}\left( X_{t} \in A, \xi_{t} \leq U \right) \\
& =\mathbb{E}\left[ \mathbb { E } \left[ \mathds{1}_{\left\{ X_{t} \in A, U \geqslant \xi_{t} \right\}} \mid \mathcal F^{x_0,W}_t  \right] \right]  \\
& =\mathbb{E}\left[ \int_{0}^{1} \mathds{1}_{\left\{ X_{t} \in A, \xi_{t} \leq u \right\}} d u \right] \\
& =\mathbb{E}\left[\int_{\xi_{t}}^{1} \mathds{1}_{\left\{X_{t} \in A\right\}} d q\right] \\
& =\mathbb{E}\left[\mathds{1}_{\left\{X_{t} \in A\right\}}\left(1-\xi_{t}\right)\right], \quad \forall A \in \mathcal B ( \R^d), \ t \in [0,T],
\end{aligned}
\end{equation}
where, again, independence of $U \sim \text{Un}(0,1)$ with respect to $(W,x_0)$ has been employed.

Similarly, for the second consistency condition in \eqref{eq def equilibrium OS MFG}, we find
$$
\begin{aligned}
\mathbb{P}\left( X_{\tau^{\xi}} \in B, \tau^{\xi}\leq t\right) 
& =\mathbb{E}\left[\mathds{1}_{\left\{X_{\tau^{\xi}} \in B, \tau^{\xi} \leq t\right\}}\right] \\
& =\mathbb{E}\left[\mathds{1}_{\left\{X_{\tau^{\xi}} \in B, \xi_{t}\geq U \right\}}\right], \quad \forall B \in \mathcal{B}(\R^d), \ t \in [0, T].
\end{aligned}
$$
Furthermore, letting
$$
{\tau^{\xi}}(u):=\inf \left\{ t \, | \,  \xi_{t}> u \right\}, \quad u \in[0,1], 
$$
we have $\tau ^\xi=\tau^{\xi}(U)$. 
Hence 
$$
\begin{aligned}
\mathbb{P}\left( X_{\tau^{\xi}} \in B, \tau^{\xi}\leq t\right) = \mathbb{E}\left[\mathds{1}_{\left\{X_{\tau^{\xi}} \in B, \xi_{t} \geq U\right\}}\right] 
= & \mathbb{E}\left[\mathbb{E}\left[\mathds{1}_{\left\{X_{\tau^\xi(U)} \in B, \xi_{t} \geq U\right\}} \mid \mathcal{F}_{t}^{x_0,W}\right]\right] \\
= & \mathbb{E}\left[\int_{0}^{1} \mathds{1}_{\left\{X_{\tau^{\xi}(u)} \in B, \xi_{t} \geq u\right\}} d u \right] \\
= & \mathbb{E}\left[\int_{0}^{1} \mathds{1}_{\left\{X_{\tau^{\xi}(u)} \in B, \tau^{\xi}(u) \leq t\right\}} d u \right]. 
\end{aligned}
$$
By using the change of variables formula (see Proposition 4.9 in Chapter 0 in \cite{revuz.yor.2013continuous}, among others), we finally obtain
\begin{equation}\label{eq consistency mu singular control}
\mathbb{E}\left[\mathds{1}_{\left\{X_{\tau^{\xi}} \in B, \xi_{t} \geq U\right\}}\right] 
 =
\mathbb{E}\left[\int_{[0,T]} \mathds{1}_{\left\{X
_{s} \in B, s \leq t\right\}} d \xi_{s}\right] 
 =\mathbb{E}\left[\int_{[0,t]} \mathds{1}_{\left\{X_{s} \in B \right\}} d \xi_{s}\right].
\end{equation}
Thus, if for $(\hat \xi, \hat m, \hat \mu) \in \A \times \mathcal M$ the triple $( \tau^{\hat \xi}, \hat m, \hat \mu)$ satisfy the consistency conditions in \eqref{eq def equilibrium OS MFG}, then by \eqref{eq consistency m singular control} and \eqref{eq consistency mu singular control} it follows that
\begin{equation}
    \label{eq consistency singular control}
    \begin{aligned}
    \hat m_{t}(A)&=\mathbb{E}\left[\mathds{1}_{\left\{X_{t} \in A\right\}} (1-\hat \xi_{t} )\right], \quad \forall A \in \mathcal{B}(\R^d), \  t \in[0, T],
    \\
\hat \mu(B \times [0, t]) & =\mathbb{E}\left[\int_{[0,t]} \mathds{1}_{\left\{X_{s} \in B \right\}} d \hat \xi_{s}\right], \quad \forall B \in \mathcal{B}(\R^d), \ t \in [0, T],
\end{aligned}
\end{equation}
which is the desired reformulation in terms of $\hat \xi$. 

\subsubsection{Entropy regularization and new equilibrium concept}

It is clear from \eqref{eq functional J for tau xi} that $\A \ni \xi \mapsto J(\xi,m,\mu) \in \R$ is linear, and it admits optimizers  of the "bang-bang" form $\xi_t=\mathds{1}_{t \geq \tau}$, for some strict optimal stopping $\tau\in \mathcal{T}$.
Hence, there is no guarantee that equilibria are of randomized type -- a phenomenon already noted in \cite{dianetti2024exploratory} in a single-agent setting without mean-field interaction.

To encourage randomized equilibrium strategies, we introduce an entropy-regularized version of the profit functional \( J \). Specifically, we consider a general entropy function \( \mathcal{E} \) that satisfies the following conditions.

{\begin{assumption}
    \label{assumption entropy}
   The entropy function $\mathcal E : [0,1] \to [0,\infty)$ is bounded, strictly concave, and has maximum in $(0,1)$. 
\end{assumption}}
A natural choice for the entropy function $\mathcal E$ is  the \emph{cumulative residual entropy}
(see \cite{rao.chen.vemuri.wang.2004cumulative}) which is given by $\mathcal E (z) = - z \log z$, employed in \cite{dianetti2024exploratory}. We also refer to \cite{han2023choquet} for examples related to Choquet integrals.

Given a temperature parameter $\lambda \geq 0$ that controls the level of exploration, for any $(m,\mu) \in \mathcal M$ we define the entropy regularized profit functional 
\begin{equation}
    \label{eq entropy regularized functional}
    J^\lambda(\xi ,  m, \mu) := \mathbb{E}\left[\int_{0}^{T}\left( f (t,X_{t}, m_{t} ) \left(1-\xi_{t}\right) + \lambda \mathcal E \left( \xi_{t}\right) \right) d t+\int_{[0,T]} g (
X_{t}, \mu) d \xi_{t}\right], \ \ \xi \in \mathcal A.
\end{equation}
A control  $ \xi^* \in \A$ is optimal  for $(m,\mu)$ if $J^\lambda( \xi^* ,  m, \mu) \geq J^\lambda(\xi ,  m, \mu)$ for any other $\xi \in \A$.
The existence of an optimal control follows by classical arguments (see e.g.\ \cite{li&zitkovic2017}), after noticing that the function $ \xi \mapsto J^\lambda(\xi , m, \mu) $ is  concave. 
In the case $\lambda>0$, such a concavity is strict, and the optimizer is unique.

The entropy $\mathcal E$ is nonnegative and achieves its highest values when the probability $\xi_t$ is near the maximum point of $\mathcal E$. 
Since the maximum point of $\mathcal E$ is an internal point of $[0,1]$, this incentivizes the use of randomized stopping times.
Indeed, in the case $\lambda >0$, it can be seen that the unique optimal control is not of bang-bang type (see also \cite{dianetti2024exploratory}).

In light of the consistency conditions derived in \eqref{eq consistency singular control}, we introduce the following notion of equilibrium to the entropy regularized MFG with singular controls ($\lambda$-SC-MFG, in short).
\begin{definition}[Entropy regularized MFG equilibrium with singular controls]
An equilibrium to the $\lambda$-SC-MFG problem  is a triple $(\hat \xi, \hat m, \hat \mu)$ such that
\begin{equation}\label{eq def equilibrium SC lambda MFG}
\begin{cases}
& \begin{matrix} \hat \xi \in \argmax_{\xi \in \mathcal{A}} J^\lambda(\xi , \hat m,\hat \mu),\end{matrix}  \\ 
& \hat m_{t}(A)=\mathbb{E}\left[ \mathds{1}_{ \{X_{t} \in A \}}(1-\hat \xi_t)\right], \quad \forall A \in \mathcal{B}\left(\R^d\right), \  t \in[0, T],\\
& 
\hat \mu(B \times [0, t]) =\mathbb{E}\left[\int_{[0,t]} \mathds{1}_{\left\{X_{s} \in B\right\}} d \hat \xi_{s}\right], \quad \forall B \in \mathcal{B}\left(\R^d\right), \  t \in[0, T]. 
\end{cases}
\end{equation}
\end{definition}

\begin{remark}
 From the mathematical point of view, we notice that:
 \begin{enumerate}
     \item 
 When $\lambda = 0$, we have $J^0(\xi ,  m, \mu)= J(\tau^\xi ,  m, \mu)$, so that  the problem \eqref{eq def equilibrium SC lambda MFG} is just the definition of equilibrium for the OS-MFG with mixed stopping strategies. Such a game is very closely related to those studied in \cite{bouveret.dumitrescu.tankov.20,dumitrescu2021control,dumitrescu.leutscher.tankov.2022linear}, using a linear programming approach. A rigorous relation between   {mixed stopping strategies}  and the occupation measures arising in the linear-programming approach is provided in \cite{roxanaetal}, where mixed strategies are expressed in terms of probability  kernels. 
   \item When $\lambda >0$, the problem \eqref{eq def equilibrium SC lambda MFG} is a new type of MFG with singular controls which has not been considered in the  literature. In particular, such a model cannot be embedded into the (general) results obtained in \cite{Campietal, denkert.horst.2023extended,  fu2023extended, Fu&Horst17} because of the structure of our consistency condition \eqref{eq consistency singular control} involving the integral with respect to the $d\xi$-random measure. 
 \end{enumerate}
\end{remark}
\begin{remark}
For any fixed \( (m, \mu) \in \mathcal{M} \), the optimal value and the corresponding optimal strategy of the control problem remain stable as \( \lambda \to 0 \) (see \cite{dianetti2024exploratory} for details). However, due to the well-known sensitivity of Nash equilibria to model perturbations, it is a nontrivial question whether the equilibria themselves remain stable in the limit \( \lambda \to 0 \).

From an applied and learning perspective, this question is particularly important, as it provides theoretical guarantees on the closeness between the equilibria of the entropy-regularized problem \eqref{eq def equilibrium SC lambda MFG} -- which we aim to learn -- and those of the original problem \eqref{eq def equilibrium OS MFG}. We address this issue and establish a positive result in Theorem~\ref{theorem qualitative stability in lambda}.

\end{remark}

\section{Analysis of entropy regularized OS-MFGs}
\label{sec:existence}

The aim of this section is to study in detail the entropy regularized $\lambda$-SC-MFG \eqref{eq def equilibrium SC lambda MFG} and its relations, as the temperature parameter $\lambda \to 0$, with our original OS-MFG \eqref{eq def equilibrium OS MFG}.

\subsection{Preliminary results}
We first endow the set of singular controls $\mathcal A$ with a topology. 
Notice that $\mathcal A$ is a closed convex subset of the Hilbert space $\mathbb L ^2 (\Omega \times [0,T])$,  i.e. the set of $\mathcal{F} \otimes \mathcal{B}([0,T])$-measurable functions $\ell$ such that $\mathbb{E}[\int_0^T \ell^2_tdt] <\infty$: 
indeed, if a sequence $(\xi^n)_n \subset \A$ converges to some $\xi \in \mathbb L ^2 (\Omega \times [0,T])$, it is always possible (e.g., via Lemmas 4.5 and 4.6 in \cite{Karatzas&Shreve84}) to find a c\`adl\`ag monotone nondecreasing modification of $\xi$, so that $\xi \in \A$.
Thus, consider on $\A$  the topology of weak convergence (in the Hilbert sense) of $\mathbb L ^2 (\Omega \times [0,T] )$, and  denote the convergence with $\xi^n \to \xi$.
Since the set $\A$ is convex and closed for the strong topology, it is closed for the weak topology. 
Therefore, since $\A$ is also bounded, it is compact when endowed with the weak topology.

We define the consistency map
$ \Gamma : \mathcal A  \to \mathcal M$ as the map $\Gamma (\xi) := (m^\xi, \mu^\xi)$ with
\begin{equation}\label{eq consistency map} 
\begin{aligned}
& m^\xi_{t}(A):=\mathbb{E}\left[\mathds{1} _{\left\{X_{t} \in A\right\}} \left(1-\xi_{t}\right)\right], \quad \forall A \in \mathcal{B}(\R^d), \quad t \in[0, T], \\
& \mu^\xi(B \times [0, t])
:=\mathbb{E}\left[\int_{[0,t]} \mathds{1}_{\left\{X_{t} \in B\right\}} d \xi_{t}\right],  \quad  \forall B \in \mathcal{B}(\R^d), \quad t \in[0, T].
\end{aligned}
\end{equation}

For our subsequent analysis, in the next three lemmas we discuss boundedness, closedness and  compactness of the set $\Gamma (\A)$.
\begin{lemma}
    \label{lemma a priori estimates}
    The consistency map $\Gamma$ is bounded; that is,
    there exists a constant $C_p >0$ such that
    $$
    \int _{\R^d} |z|^p m_t^\xi(dz) \leq C_p ,\ \ \forall t \in [0,T],
    \and
    \int _{[0,T] \times \R^d} (t^p+|z|^p) \mu^\xi(dt,dz) \leq C_p,
    $$
    for any $ \xi \in \A$.
\end{lemma}
\begin{proof}
By using the boundedness condition $0\leq \xi_t \leq 1$,  we obtain
$$
\int _{\R^d} |z|^p m_t^\xi(dz) = \E [ |X_t|^p (1-\xi_t)] \leq \E \Big[ \sup_{t \in [0,T] } |X_t|^p \Big] \leq C_p,   
$$
for a suitable constant $C_p$ (see \eqref{eq estimate SDE}). This shows the first estimate. 
For the second estimate, again by boundedness of $\xi$, we have
$$
\int _{[0,T] \times \R^d} (t^p+|z|^p) \mu^\xi(dt,dz) = \E \left[ \int _{[0,T]} (t^p+|X_t|^p) d\xi_t \right] \leq T^p+\E \Big[ \sup_{t \in [0,T] } |X_t|^p \Big] \leq C_p,
$$
unless to take a larger $C_p$.
\end{proof}
\begin{lemma}
    \label{lemma closedness Gamma A}
   The consistency map $\Gamma: \A \to \mathcal M$ has closed graph.
\end{lemma}
\begin{proof}
We have to show that, if $(m^n,\mu^n) \in \Gamma(\mathcal{A})$ converges to $(m,\mu)$, then $(m,\mu) \in \Gamma(\mathcal{A})$. Let $\xi^n \in \mathcal{A}$ such that $(m^n,\mu^n)=(m^{\xi^n},\mu^{\xi^n})$. By compactness of the set $\mathcal{A}$, there exists a subsequence $\xi^{n_k}$ which converges weakly to $\xi^\star$.
Since $\sup_{t \geq 0} \xi^{n_k}_t \leq 1$, we can employ  Lemma 3.5 in \cite{K} in order to find a subsequence (not relabelled)  $(n_k)_k$ and a limit point $\xi \in \mathcal A$ such that, setting
$
\zeta^j_t := \frac1j \sum_{k=1}^j \xi^{n_k}_t,
$
we have
$$
\int_{[0,T]}
\phi_t d \zeta^j_t \to \int_{[0,T]} \phi_t d \xi_t,  \ \mathbb P \text{-a.s., as $j \to \infty$, for any $\phi \in C_b([0,T])$.}
$$
One can easily check that $\xi=\xi^\star$.
Therefore, it holds
$$
\begin{aligned}
\int_{\mathbb{R}^d}\varphi(x)m_{t}(dx)
&= \lim_n  \frac1N \sum _{k=1}^N \int_{\mathbb{R}^d}\varphi(x)m_{t}^{n_k}(dx)  \\
&= \lim_n  \frac1N \sum _{k=1}^N
\mathbb{E}\left[\varphi ( X_{t} ) \left(1-\xi^{n_k}_{t}\right)\right] \\
&= \lim_n  
\mathbb{E}\left[\varphi ( X_{t} ) \left(1-\zeta^{n}_{t}\right)\right] \\
& = \mathbb{E}\left[\varphi ( X_{t} ) \left(1-\xi_{t}\right)\right]
, \quad \forall \  t \in[0, T], \quad \varphi \in C_b (\mathbb R ^d ), 
\end{aligned}
$$
and
$$
\begin{aligned}
\int_{[0,T] \times \mathbb{R}^d}\varphi(t,x)\mu(dt,dx)
&= \lim_n  \frac1N \sum _{k=1}^N \mathbb{E}\bigg[\int_{[0,T]} \varphi (t, X_{t} ) d \xi^{n_k}_{t}\bigg]  \\
&= \lim_n  
\mathbb{E}\bigg[\int_{[0,T]} \varphi (t, X_{t} ) d \zeta^N_{t}\bigg] \\
& = \mathbb{E}\bigg[\int_{[0,T]} \varphi (t, X_{t} ) d \xi_{t}\bigg]
, \quad \forall \   \varphi \in C_b ([0,T] \times \mathbb R ^d ),
\end{aligned}
$$
from which we conclude that $(m,\mu)=\Gamma(\xi)$, thus completing the proof.
\end{proof}

\begin{lemma}
    \label{lemma compactness Gamma A}
   The set $\Gamma (\A) \subset \mathcal M$ is compact.
\end{lemma}
\begin{proof}

Let us first introduce the set $\mathcal{M}^\star$
of pairs $(\mu, m) \in \mathcal{P}_p([0, T]\times \mathbb{R}^d)\times M_p$, such that for all $u\in C_b^{1, 2}([0, T]\times \mathbb{R}^d)$,
\begin{equation*}
\int_{[0, T]\times {\mathbb{R}^d}} u(t, x)\mu(dt, dx)= \int_{\mathbb{R}^d} u(0, x)m_0(dx) + \int_0^T \int_{_{\mathbb{R}^d}} \left(\partial_t u +\mathcal L u\right) (t, x)m_t(dx)dt,
\end{equation*}
where $\mathcal L u(t, x):= \nabla_x u(t, x)^\top b(t, x) + \frac{1}{2}Tr[\sigma^{\top}(H_xu)\sigma]$, with $\nabla_x u:=(\partial_{x_1}u, \ldots, \partial_{x_d}u)$, $H_xu$ the Hessian matrix of $u$ with respect to $x$ and $Tr$ the trace operator. We also denote by $m_0:=\mathcal{L}(x_0)$.
For a given $(\xi_t) \in \mathcal{A}$, by applying Ito's formula, it can be shown that the associated pair $(m^{\xi},\mu^{\xi})=\Gamma(\xi)$  belongs to $\mathcal{M}^\star$. Therefore, $\Gamma(\mathcal{A}) \subset \mathcal{M}^{\star}$.
By Theorem 2.10 in \cite{dumitrescu.leutscher.tankov.2022linear}, the set $\mathcal{M}^\star$ is compact for the topology of the convergence in measure.\\
Furthermore, thanks to Lemma \ref{lemma closedness Gamma A}, the set $\Gamma(\mathcal{A})$ is a closed subset of $\mathcal{M}^\star$. It thus follows that it is compact.
\end{proof}
\begin{remark} Under some additional assumptions (e.g. uniform ellipticity) it can be actually shown that $\Gamma(\mathcal{A})=\mathcal{M}^*$ (see \cite{roxanaetal}). 
\end{remark}

Next, for $(m,\mu,\lambda)  \in \mathcal M \times [0,\infty)$, define the best response map $R_\lambda (m,\mu) \subset \mathcal A$ as
\begin{equation}
    \label{eq best response map}
    R_\lambda (m,\mu) := \argmax_{\xi \in \mathcal{A}} J^\lambda(\xi ,  m, \mu).
\end{equation}
For any $(m,\mu,\lambda)  \in \mathcal M \times [0,\infty)$, via classical concavity arguments (see e.g. \cite{li&zitkovic2017}), we have that $R_\lambda (m,\mu)$ is nonempty.
Notice that, when $\lambda >0$, the strict concavity of $\xi \mapsto J^\lambda(\xi ,  m, \mu)$ implies that $R_\lambda (m,\mu)$ is actually a singleton. 
In this case, with abuse of notation we will indicate with $ R_\lambda (m,\mu) = (R_\lambda (m,\mu)_t)_t$ both the singleton and its only element.

The best response map has closed graph, and  enjoys the following continuity properties.
\begin{lemma}
    \label{lemma continuity best response}
     Under Assumptions \ref{assumption}, the following statements hold true:
    \begin{enumerate}
        \item\label{lemma continuity best response: one}  The best response map $(m,\mu,\lambda) \mapsto R_\lambda (m,\mu) $ is continuous in $\Gamma(\A) \times (0,\infty).$
        \item\label{lemma continuity best response: two} For any  sequence $(m^n,\mu^n,\lambda^n)_n \subset \Gamma(\A) \times [0,\infty)$ converging to $(m,\mu,0)$, and any $\xi^n \in R_{\lambda^{n}}(m^{n},\mu^{n})$, there exists a subsequence $(n_k)_k$ and $\xi^* \in R _0 (m,\mu)$ such that $\xi^{n_k} \to\xi^*$. 
    \end{enumerate}
\end{lemma}
\begin{proof}
Consider a sequence $(m^n,\mu^n,\lambda^n)_n$ converging to $(m,\mu,\lambda)$ and take $\xi^n \in R_{\lambda^n}(m^n,\mu^n)$.
By compactness of $\A$ in the weak topology,  for any subsequence of indexes  $(n_k)_k$, we can find a further subsequence (not relabeled)  and a limit point $\xi^*$ such that $\xi^{n_k} \to \xi^*$ weakly in $\mathbb L ^2 (\Omega \times [0,T] )$ as $k \to \infty$.
    We need to show that $\xi^*$ is optimal for $J^\lambda (\cdot, m , \mu)$.
    To this end, notice that, by optimality of $\xi^{n_k}$ for $J^{\lambda^{n_k}} (\cdot, m^{n_k} , \mu^{n_k})$, we have 
    \begin{equation}
        \label{eq lemma continuity optimality prelimit}
        J^{\lambda^{n_k}} (\xi^{n_k}, m^{n_k} , \mu^{n_k}) \geq J^{\lambda^{n_k}} (\xi, m^{n_k} , \mu^{n_k}), \quad \forall \xi \in \A .
    \end{equation}
    In order to discuss limits as $n_k \to \infty$ in the previous inequality, we first show that
    \begin{equation}
        \label{eq lemma continuity zero limit}
        \lim_k \sup_{\xi\in\A} | J^{\lambda^{n_k}} (\xi, m^{n_k} , \mu^{n_k}) - J^{\lambda } (\xi, m  , \mu)|=0.
    \end{equation}
    Using the boundedness of $\mathcal E$ and of the controls $\xi$, we first find
    $$
    \begin{aligned}
    &\sup_{\xi\in\A} |J^{\lambda^{n_k}} (\xi, m^{n_k} , \mu^{n_k}) - J^{\lambda } (\xi, m  , \mu) | \\
    & \quad \leq \sup_{\xi\in\A} 
    \mathbb{E}\left[\int_{0}^{T}\left| f (t,X_{t}, m_{t}^{n_k} ) - f (t,X_{t}, m_{t} ) \right) \left(1-\xi_{t} )\right| d t+\int_{[0,T]} \left| g (
X_{t}, \mu^{n_k}) - g (
X_{t}, \mu)  \right| d \xi_{t}\right] 
\\
    & \quad \quad + |\lambda^{n_k} -  \lambda| \, \sup_{\xi\in\A}  \mathbb{E}\left[\int_{0}^{T} \mathcal E (\xi_t) dt \right] \\ 
 & \quad \leq 
    \mathbb{E}\left[ \int_{0}^{T} \left| f (t,X_{t}, m_{t}^{n_k} ) - f (t,X_{t}, m_{t} ) \right| dt \right] + \mathbb{E}\left[ \sup_{ t \in [0,T]} \left| g (
X_{t}, \mu^{n_k}) - g (
X_{t}, \mu)  \right|  \right] + C|\lambda^{n_k} -  \lambda|,
    \end{aligned}
    $$
    so that it remains to study the convergence of the  two terms under expectation. 
    For the term involving $f$, we show that for any subsequence of $(n_k)_k$ (not relabelled), there exists a further subsequence (again, not relabelled) such that 
    \begin{equation}\label{eq limits f}
    \lim_k  \mathbb{E}\left[ \int_{0}^{T} \left| f (t,X_{t}, m_{t}^{n_k} ) - f (t,X_{t}, m_{t} ) \right| dt \right]=0,  
    \end{equation}
    so that the whole sequence converges. 
    Indeed, given a subsequence, we can find a further subsequence such that $d^\prime_p (m_t ^{n_k}, m_t) \to 0$ $dt$-a.e. (see \eqref{eq conv in dt after conv in measure}), and thanks to the continuity of $f$, we have that $f (t,X_{t}, m_{t}^{n_k} ) \to f (t,X_{t}, m_{t} )$ $\mathbb P \otimes dt$-a.e.
    Moreover, the growth conditions of $f$ give
    $$
    \begin{aligned}
    \sup_k  \left| f (t,X_{t}, m_{t}^{n_k} ) - f (t,X_{t}, m_{t} ) \right| 
    & \leq C \Big( 1  +  \sup_{t \in [0,T]} |X_t|^p  + \sup_k \int_{\R^d} |z|^p  m^{n_k}_t(dz) +  \int_{\R^d} |z|^p  m_t(dz) \Big) \\
    & \leq C \Big( 1  +  \sup_{t \in [0,T]} |X_t|^p  + 2 C_p \Big), 
    \end{aligned}
    $$
    for a constant $C_p$ as in Lemma \ref{lemma a priori estimates} (notice indeed that the estimate $  \int_{\R^d} |z|^p  m_t(dz)$ easily follows by  convergence).
    Since $\E \left[ \sup_{t \in [0,T]} |X_t|^p \right]$ is finite (see \eqref{eq estimate SDE}), the limits in \eqref{eq limits f} follows by the dominated convergence theorem. 
    We next show the limit for the term involving $g$. 
    Using the uniform continuity of $g$, for $w_g$ as in Assumption \ref{assumption} we have 
    \begin{equation*}
        \lim_k \mathbb{E}\left[ \sup_{ t \in [0,T]} \left| g (X_{t}, \mu^{n_k}) - g (X_{t}, \mu) \right|  \right] 
        \leq  \lim_k  C \Big( 1  + \mathbb{E} \Big[ \sup_{t \in [0,T]} |X_t|^p  \Big] \Big) \,  w_g (\mu^{n_k}, \mu)  = 0.  
    \end{equation*}
    This limit, together with \eqref{eq limits f}, allows us to obtain  \eqref{eq lemma continuity zero limit}.

    As a consequence of \eqref{eq lemma continuity zero limit},   we have 
    \begin{equation}
        \label{eq lemma continuity first limit}
        \lim_k J^{\lambda^{n_k}} (\xi, m^{n_k} , \mu^{n_k}) = J^{\lambda } (\xi, m  , \mu),
    \end{equation}
    as well as 
    \begin{equation}
        \label{eq lemma continuity second limit}
        \lim_k |J^{\lambda^{n_k}} (\xi ^{n_k}, m^{n_k} , \mu^{n_k}) - J^{\lambda } (\xi^{n_k}, m  , \mu)| =0. 
    \end{equation}

    We next want to show that
    \begin{equation}
        \label{eq lemma continuity third limit}
        V^* := \limsup_k J^{\lambda } (\xi^{n_k}, m  , \mu) \leq J^{\lambda } (\xi^*, m  , \mu).
    \end{equation}
    Fix a  subsequence (not relabelled)  $(n_k)_k$ such that $ V^* := \lim_k J^{\lambda } (\xi^{n_k}, m  , \mu)$. 
    Since $\sup_{t \geq 0} \xi^{n_k}_t \leq 1$, we can again employ  Lemma 3.5 in \cite{K} in order to find a subsequence (not relabelled)  $(n_k)_k$ and a limit point $\bar \xi \in \mathcal A$ such that, setting
$
\zeta^j_t := \frac1j \sum_{k=1}^j \xi^{n_k}_t,
$
we have
$$
\int_{[0,T]}
\phi_t d \zeta^j_t \to \int_{[0,T]} \phi_t d \bar \xi_t,  \ \mathbb P \text{-a.s., as $j \to \infty$, for any $\phi \in C_b([0,T])$.}
$$
The processes  $\bar \xi$ and $\xi^*$ coincide. 
Furthermore, by the convergence of $\zeta^j$ and the concavity of $J^\lambda (\cdot, m,\mu)$ we find
$$
\begin{aligned}
    J^{\lambda } (\xi^*, m  , \mu) = \lim_j J^\lambda (\zeta^j , m,\mu) &\geq \lim_j \frac{1}{j} \sum_{k=1}^j J^\lambda (\xi^{n_k} , m,\mu) = V^*,
\end{aligned}
$$
which proves \eqref{eq lemma continuity third limit}.

Finally, by using \eqref{eq lemma continuity first limit}, \eqref{eq lemma continuity second limit} and \eqref{eq lemma continuity third limit} to take limits as $k \to \infty$ in \eqref{eq lemma continuity optimality prelimit}, we obtain
$$
J^{\lambda } (\xi^*, m  , \mu ) 
\geq 
J^{\lambda } (\xi, m  , \mu), \quad \forall \xi \in \A, 
$$
showing that $\xi^* \in R_\lambda (m,\mu)$.
When $\lambda =0$, this completes the proof of Claim \eqref{lemma continuity best response: two}.

To show Claim \eqref{lemma continuity best response: one}, we notice that, when $\lambda >0$, by uniqueness of the optimal control we have that $\xi^* = R_\lambda (m,\mu)$.
In this case, we have shown that for any subsequence  $(n_k)_k$ we can find a further subsequence (not relabelled) such that $(\xi^{n_k})_k$ converges to  $R_\lambda (m,\mu)$. 
Hence, the whole sequence $(\xi^n)_n$ converges to $R_\lambda (m,\mu)$, thus completing the proof.
\end{proof}

\subsection{Existence of equilibria of the SC-MFG}
We are now ready to discuss the following existence result.
\begin{theorem}
    \label{theorem existence MFGE for SC-MFG}
    Under Assumption \ref{assumption}, for any $\lambda\geq 0$ there exists an equilibrium to the $\lambda$-SC-MFG \eqref{eq def equilibrium SC lambda MFG}.
\end{theorem}

\begin{proof}
The proof consists in employing Kakutani–Glicksberg–Ky Fan fixed point theorem to the map $\mathcal R _\lambda: \mathcal A \to \mathcal A$, given by the composition $\mathcal R_\lambda := R_\lambda \circ \Gamma $. 
Noticing that $\mathcal A$ is a nonempty compact convex subset of the  locally convex Hausdorff topological vector space $\mathbb L ^2 ( \Omega \times [0,T])$, it only remains to show that the map $\mathcal R _\lambda$ has closed graph.

To this end, for any sequence $\xi^n \to \xi$, with
$\mathcal R _\lambda (\xi^n) \to  \xi^\lambda \in \mathcal A $, we need to show that $\xi^\lambda \in \mathcal R _\lambda (\xi)$.
Set
$ (m^n, \mu^n):= (m^{\xi^n}, \mu^{\xi^n}):=\Gamma(\xi^n)$ and notice that, due to the compactness of $\Gamma (\mathcal A )$ (see Lemma \ref{lemma compactness Gamma A}), there exists a subsequence (not relabeled) 
and a limit point $(m, \mu) \in \mathcal M$ such that $(m^n, \mu^n) \to (m, \mu)$. 
Moreover, by Lemma \ref{lemma closedness Gamma A}, we have
\begin{equation}
    \label{eq inside exist}
    (m, \mu) = \Gamma (\xi).
\end{equation}
Now, by assumption, $\mathcal R _\lambda (\xi^n) =  R _\lambda ( \Gamma (\xi^n))$ converges to $\xi^\lambda$ and,
using \eqref{eq inside exist}, we have
$\Gamma (\xi^{n_k}) \to \Gamma (\xi)$ on a subsequence $n_k$. 
Thanks to Lemma \ref{lemma continuity best response}, the map $R_\lambda : \mathcal M \to \mathcal A$ has closed graph, so that
$$
\xi^\lambda \in R_\lambda ( \Gamma (\xi) ) = \mathcal R _\lambda (\xi).
$$
This shows that $\mathcal R _\lambda$ has closed graph, and hence completes the proof.
\end{proof}

\subsection{$\lambda$-stability of equilibria}
A crucial point when studying the limit, as $\lambda \to 0$, of the equilibria to the entropy regularized $\lambda$-SC-MFG \eqref{eq def equilibrium SC lambda MFG} is whether the equilibrium to the $0$-SC-MFG \eqref{eq def equilibrium SC lambda MFG} is unique or not. 

To this end, we introduce the following requirement, which is a version of the so called Lasry-Lions monotonicity conditions. 

\begin{assumption}[Monotonicity]\label{assumption monotonicity} 
For any $\lambda \geq 0$, the following hold true:
    \begin{enumerate}
        \item\label{assumption monotonicity: uniqueness R} For any $(m, \mu)$,  $R_\lambda (m,\mu)$ is a singleton.
        \item\label{assumption monotonicity: LL}  
        $
        J^{\lambda}(\xi, m^{\xi},\mu^{\xi}) - J^{\lambda}(\bar \xi, m^{ \xi},\mu^{\xi}) - ( J^{\lambda}(\xi, m^{\bar \xi},\mu^{\bar{\xi}}) - J^{\lambda}(\bar \xi, m^{\bar \xi}, \mu^{\bar{\xi}})) \leq  0.
        $
    \end{enumerate}
\end{assumption}
Notice that Condition \ref{assumption monotonicity: uniqueness R} in Assumption \ref{assumption monotonicity} is always satisfied when $\lambda >0$.
\begin{remark} We provide here an example in which Condition \eqref{assumption monotonicity: LL} in Assumption \eqref{assumption monotonicity} is satisfied. For example, one could consider $f$ and $g$ of the form
\begin{align*}
f(t,x,m)=\bar{k}(x)\bar{f}\left(t,\int_{\mathbb{R}^d}\bar{k}(x)m_t(dx) \right),
\quad
g(x,\mu)=\bar{\ell}(x)\bar{h}\left(\int_{[0,T] \times \mathbb{R}^d}\bar{\ell}(x)\mu(dt,dx) \right),
\end{align*}
with $\bar{f}$ non-increasing in the second argument and $\bar{h}$ non-decreasing.
\end{remark}
\begin{theorem}
    \label{theorem uniqueness MFGE}
    Under Assumptions \ref{assumption} and \ref{assumption monotonicity}, for any $\lambda \geq 0$ there exists a unique equilibrium to $\lambda$-SC-MFG \eqref{eq def equilibrium SC lambda MFG}.
\end{theorem} 
\begin{proof}
The proof of this result this theorem is standard and slightly adapted from \cite{bouveret.dumitrescu.tankov.20}, we include it for completeness.

Arguing by contradiction, assume that there exists two distinct mean-field equilibria $(\xi, m, \mu)$, $(\bar \xi, \bar m, \bar \mu)$.
By uniqueness of the optimal control and the definition of equilibrium, we have 
$$
\begin{aligned}
    J^{\lambda}(\xi, m^{\xi}, \mu^{\xi}) - J^{\lambda}(\bar \xi, m^{ \xi}, \mu^{\xi}) &>0
    \and
    J^{\lambda}(\bar \xi, m^{\bar \xi}, \mu^{\bar \xi}) - J^{\lambda}( \xi, m^{\bar \xi}, \mu^{\bar \xi})>0.
\end{aligned}
$$
Summing up these two inequalities, we obtain
$$
     J^{\lambda}(\xi, m^{\xi}, \mu^{\xi}) - J^{\lambda}(\bar \xi, m^{ \xi}, \mu^{\xi}) - ( J^{\lambda}(\xi, m^{\bar \xi} , \mu^{\bar \xi}) - J^{\lambda}(\bar \xi, m^{\bar \xi}, \mu^{\bar \xi})) >0,
$$
which contradicts the second condition in Assumption \ref{assumption monotonicity}.
Therefore, the equilibrium is unique.
\end{proof}

Entropy regularization and the corresponding temperature parameter $\lambda$ are introduced to  encourage randomization in a learning environment. It is important to understand the  closeness between the equilibria of the entropy-regularized problem \eqref{eq def equilibrium SC lambda MFG} -- which we aim to learn -- and those of the original problem \eqref{eq def equilibrium OS MFG}. We discuss this convergence in the following theorem.

\begin{theorem}
    \label{theorem qualitative stability in lambda}
    Under Assumption \ref{assumption}, for any $\lambda > 0$, let $(\hat \xi ^\lambda, \hat m ^\lambda, \hat \mu ^\lambda)$ be an equilibrium to the $\lambda$-SC-MFG.
    Then
    \begin{enumerate}
        \item\label{theorem qualitative stability in lambda: one} $(\hat \xi ^\lambda, \hat m ^\lambda, \hat \mu ^\lambda) \to (\hat \xi ^0, \hat m ^0, \hat \mu ^0)$ up to subsequence as $\lambda \to 0$, where $(\hat \xi ^0, \hat m ^0, \hat \mu ^0)$ is an equilibrium to the $0$-SC-MFG;
        \item\label{theorem qualitative stability in lambda: two}  Under the additional Assumption \ref{assumption monotonicity}, we have $(\hat \xi ^\lambda, \hat m ^\lambda, \hat \mu ^\lambda) \to (\hat \xi ^0, \hat m ^0, \hat \mu ^0)$ as $\lambda \to 0$. 
    \end{enumerate}
\end{theorem}
\begin{proof}
To simplify the notation, we simply write $(\xi^\lambda, m ^\lambda, \mu ^\lambda)$ instead of $(\hat \xi ^\lambda, \hat m ^\lambda, \hat \mu ^\lambda)$, for any $\lambda \geq 0$.

Fix any sequence $(\lambda_n)_n \subset (0,\infty)$ with $\lambda _n \to 0$ as $n \to \infty$ and set 
$$
( \xi ^n,  m ^n,  \mu ^n) 
:= ( \xi ^{\lambda_n},  m^{\lambda_n},\mu^{\lambda_n}).
$$
Since $( \xi ^n,  m ^n,  \mu ^n)$ are assumed to be equilibria, we have $ (m ^n,  \mu ^n) = \Gamma (R_\lambda (\xi^n))$, so that $(m ^n,  \mu ^n)_n \subset \Gamma (\mathcal A)$.
By compactness of $\mathcal A \times \Gamma (\mathcal A)$  (see Lemma \ref{lemma compactness Gamma A}), we can extract a subsequence (not relabeled) of $(\xi^n,m^n,\mu^n)_n$ and a limit point $(\xi^0, m^0,\mu^0)$ such that $(\xi^n, m^n,\mu^n) \to (\xi^0, m^0,\mu^0)$.
Thanks to Lemma \ref{lemma continuity best response: two}, we have that $\xi^0 \in  R _0 (m^0,\mu^0)$. 
Moreover, by repeating the arguments in the proof of Theorem \ref{theorem existence MFGE for SC-MFG}, we have that $\Gamma (\xi^0) = (m^0,\mu^0)$. 
This, in turn implies that $\xi^0 \in  R_0 (m^0,\mu^0) =  R_0 (\Gamma (\xi^0))=\mathcal R_0 (\xi^0)$, so that $(\xi^0, m^0,\mu^0)$ is an equilibrium of the $0$-SC-MFG.
This completes the proof of Claim \eqref{theorem qualitative stability in lambda: one}. 

When  the additional Assumption \ref{assumption monotonicity} holds, by Theorem \ref{theorem uniqueness MFGE} we have the uniqueness of the equilibrium of the $0$-SC-MFG.
Hence, by the previous argument, any sequence $( \xi ^{\lambda_n},  m^{\lambda_n},\mu^{\lambda_n})$ converges to $(\xi^0, m^0,\mu^0)$, thus proving Claim \eqref{theorem qualitative stability in lambda: two}.
\end{proof}

\section{Fictitious play algorithms}
\label{sec:fictplay}



In this section, we introduce two novel fictitious play algorithms for computing the unique mean-field equilibrium $\lambda$-SC-MFG in the case when $\lambda>0$, and we establish their convergence under different data assumptions.

Note that the convergence of iterative numerical schemes -- such as fictitious play algorithms under known model parameters -- serves as a fundamental foundation for analyzing the convergence of RL algorithms in environments with unknown parameters. In many popular RL methods, such as Q-learning and actor-critic algorithms \cite{sutton1998reinforcement}, the known quantities in these schemes (e.g., the value function) are simply replaced by their corresponding estimates, computed from available observations.\\

\vspace{3mm}
\begin{algorithm}[H]\label{algo 1}
\SetAlgoLined
\KwData{A number of steps $n$ for the equilibrium approximation; a control $\bar{\xi}^{0} \in \mathcal{A}$ and $(\bar \mu^{0}, \bar m^{0}):=\Gamma(\bar{\xi}^{0}) \in \mathcal{R}_{\lambda}$;}
\KwResult{Approximate  MFG Nash equilibrium}
\For{$k=0, 1, \ldots, n-1$}{
Set $\xi^{k+1}:=\argmax_{\xi \in \mathcal{A}}{J}^\lambda (\xi,\bar m^{k},\bar \mu^{k})$;
\\
Set $\bar{\xi}^{k+1}:=\frac{k}{k+1}\bar \xi^{k} + \frac{1}{k+1}\xi^{k+1}=\frac{1}{k + 1}\sum_{\ell=1}^{k+1}\xi^{\ell}$; \\
Set $(\bar{\mu}^{k+1}, \bar{m}^{k+1}) :=\Gamma(\bar{\xi}^{k+1})$.
}
\caption{Fictitious play algorithm}
\end{algorithm}

Observe that, due to the linear dependence of the measure $(m,\mu)$ with respect to $\xi$, we have that $(\bar{\mu}^{n}, \bar{m}^{n})=\frac{1}{n} \sum_{k=1}^{n}({\mu}^{\xi^k}, {m}^{\xi^k})$.

In the next two subsections, we establish the convergence of the fictitious play algorithm $(\bar{m}^n, \bar{\mu}^n)_n$ to the equilibria of the $\lambda$-SC-MFG, under suitable structural conditions on the data, and for suitable choice of the initialization.

\subsection{Fictitious play under the Lasry-Lions condition.}

We start with the fictitious play algorithm under the Lasry-Lions monotonicity condition. 

This fictitious-play algorithm is novel in the context of SC-MFGs and also brings new contributions compared to the fictitious-play algorithm developed in \cite{dumitrescu.leutscher.tankov.2022linear} for OS-MFGs. In particular, it extends to cases where the profit-functional is non-linear with respect to the \textit{randomized stopping}, a setting in which the linear programming approach developed in \cite{dumitrescu.leutscher.tankov.2022linear} fails.

To simplify the presentation, we set up some notation. First, notice that for a given $(m,\mu) \in \Gamma({\mathcal{A}})$ and $\xi \in \mathcal{A}$, by using the identification between controls and associated measures given by $\eqref{eq consistency map}$,
we can rewrite $J^\lambda$ as follows:
\begin{align*}
J^\lambda(\xi,m,\mu)=\int_0^T \int_{\mathbb{R}^d} \left(f(t,x,m_t)m_t^{\xi}(dx)dt+g(t,x,\mu)\mu^{\xi}(dt,dx)\right)+\lambda \mathbb{E}\int_0^T {\mathcal{E}}(\xi_t)dt.
\end{align*}

We now introduce the bilinear form notation 
$$\langle f(m), m' \rangle:=\int_0^T\int_{\mathbb{R}^d}f(t, x, m_t)m_t'(dx)dt,\quad \langle g(\mu), \mu' \rangle:=\int_{[0, T]\times \mathbb{R}^d}g(t, x, \mu)\mu'(dt, dx),
$$
where $(\mu, m), (\mu', m')\in \mathcal{P}_p([0, T]\times \mathbb{R})\times M_p$.
Therefore, $J^\lambda$ writes as follows:
\begin{align*}
J^\lambda(\xi,m,\mu)=\langle f(m),m^\xi \rangle +\langle g(\mu),\mu^\xi \rangle+\lambda \mathbb{E}\int_0^T {\mathcal{E}}(\xi_t)dt.
\end{align*}

\noindent We assume the following conditions hold:
\begin{assumption}\label{main assump OS}
There exist constants $c_f > 0$ and $c_g > 0$ such that for all $t\in [0, T]$, $x, x'\in \mathbb{R}^d$, $m, m'\in \mathcal{P}^{sub}_p(\mathbb{R}^d)$, $\mu, \mu'\in \mathcal{P}_p([0, T]\times \mathbb{R}^d)$, 
$$\left|f(t, x, m) - f(t, x, m')\right|\leq c_f(1+|x|) \int_{\mathbb{R}^d}(1+|z|^p)|m-m'|(dz),$$
$$|f(t, x, m) - f(t, x, m') - f(t, x', m) + f(t, x', m')|\leq c_f |x-x'|\int_{\mathbb{R}^d}(1+|z|^p)|m-m'|(dz),$$ 
$$|g(t, x, \mu) - g(t, x, \mu')|\leq c_g (1+|x|) \int_{[0, T]\times \mathbb{R}^d}(1+|z|^p)|\mu-\mu'|(ds, dz),$$
$$|g(t, x, \mu) - g(t, x, \mu') - g(t', x', \mu) + g(t', x', \mu')|\leq c_g (|t-t'|+|x-x'|)\int_{[0, T]\times \mathbb{R}^d}(1+|z|^p)|\mu-\mu'|(ds, dz),$$
\end{assumption}
\noindent {where $|m|$ denotes the total variation measure of $m$.}  

We recall that, by Lemma \ref{lemma compactness Gamma A}, the set $\Gamma(\mathcal{A}) \subset \mathcal{M}^\star$. By Proposition 2.13 in \cite{dumitrescu.leutscher.tankov.2022linear}, the topology on the set $\mathcal{M}^\star$ is induced by the metric $\rho((m,\mu),(m',\mu')):=d_1(\mu,\mu')+d_1^M(m,m')$.

We now give the main convergence result of this section.

\begin{theorem}\label{theorem fictitious monotone}
Under Assumptions \ref{assumption}, \ref{assumption monotonicity} and \ref{main assump OS}, for any initialization $(\bar \xi^0, \bar m^0, \bar \mu^0)$, the sequence $(\bar \xi^n, \bar{m}^n, \bar{\mu}^n)_n$ defined in Algorithm 1 converges to the unique equilibrium  $(\hat{\xi}^\lambda, \hat m^\lambda, \hat \mu^\lambda)$.
\end{theorem}

\begin{proof} 
The proof is organized in two steps.\\
\textit{Step 1.} We first introduce the \textit{exploitability errors} $(\varepsilon_n)_{n\geq 1}$ which are defined as follows:
\begin{align}\label{phi n OS}
\varepsilon_n&:= J^\lambda(\xi^{n+1},\bar{m}^n,\bar{\mu}^n)-J^\lambda(\bar{\xi}^n,\bar{m}^n,\bar{\mu}^n)\nonumber\\
&=\langle f(\bar m^n),m^{n+1}-\bar m^n\rangle 
+ \langle g(\bar \mu^n),\mu^{n+1}-\bar \mu^n\rangle+ \lambda \mathbb{E}\left[\int_0^T(\mathcal{E}(\xi_t^{n+1})-\mathcal{E}(\bar{\xi}_t^n))dt \right].    
\end{align}
It can be easily observed that, by construction, $(\bar{m}^n,\bar{\mu}^n)$ is a $\varepsilon_n$-MFG equilibria. We first show here that $\varepsilon_n \to 0$.
By the optimality of $\xi^{n+1}$, it is easy to observe that $\varepsilon_n \geq 0$, for all $n \geq 1$.

\noindent We write $\varepsilon_{n+1}-\varepsilon_{n}=\varepsilon_{n}^{(1)}+\varepsilon_{n}^{(2)}$,
where 
\begin{align}
\varepsilon_n^{(1)}&:= \langle f(\bar{m}^n),\bar{m}^n \rangle +\langle g(\bar{\mu}^n),\bar{\mu}^n \rangle- \langle f(\bar{m}^{n+1}),\bar{m}^{n+1} \rangle -\langle g(\bar{\mu}^{n+1}),\bar{\mu}^{n+1} \rangle \nonumber\\
&+\lambda \mathbb{E}\left[\int_0^T (\mathcal{E}(\bar{\xi_t}^n)-\mathcal{E}({\bar{\xi}_t}^{n+1}))dt\right],
\end{align}
\begin{align}
\varepsilon_n^{(2)}&:= \langle f(\bar{m}^{n+1}),m^{n+2} \rangle+\langle g(\bar{\mu}^{n+1}),\mu^{n+2} \rangle -\langle f(\bar{m}^n),m^{n+1} \rangle- \langle g(\bar{\mu}^n),\mu^{n+1} \rangle  \nonumber \\ &+\lambda \mathbb{E}\left[\int_0^T (\mathcal{E}(\xi_t^{n+2})-\mathcal{E}(\xi_t^{n+1})))dt\right].
\end{align}
In the following $C>0$ denotes a constant (independent of $n$) that may vary from line to line. We first estimate $\varepsilon_n^{(1)}$.
First, by the concavity of the entropy functional $\mathcal{E}(\cdot)$,
we get 
\begin{align}
\mathbb{E}\left[\int_0^T\left(\mathcal{E}(\bar{\xi}_t^n)-\mathcal{E}(\bar{\xi}_t^{n+1})\right) dt\right]&=\mathbb{E}\left[\int_0^T\left(\mathcal{E}(\bar{\xi}_t^n)-\mathcal{E}\left(\frac{n}{n+1}\bar{\xi}_t^n+\frac{1}{n+1}\xi^{n+1})\right)\right) dt\right] \nonumber \\ &\leq \frac{1}{n+1}\mathbb{E}\left[\int_0^T\left(\mathcal{E}(\bar{\xi}_t^n)-\mathcal{E}(\xi_t^{n+1})\right) dt\right].
\end{align}

We recall now the following estimate (see Lemma 2.14 in  \cite{dumitrescu.leutscher.tankov.2022linear}):
there exists a constant $C_2 > 0$ (independent of $n$) such that for all $t \in [0,T]$:
\begin{align}
\int_{[0,T] \times \mathbb{R}^d} (1+|x|^p)|\bar{\mu}^{n+1}-\bar{\mu}^n|(dt,dx) \leq \frac{C_2}{n},\,\,\, \int_{[0,T] \times \mathbb{R}^d} (1+|x|^p)|\bar{m}_t^{n+1}-\bar{m}_t^n|(dt,dx) \leq \frac{C_2}{n}.
\end{align}
We then derive, for all $t \in [0,T]$ and $x \in \mathbb{R}^d$,
\begin{align*}
|f(t,x,\bar m_t^{n+1})-f(t,x,\bar m_t^n)|  &\leq c_f (1+|x|)\int_{\mathbb{R}^d}(1+|z|^p)|\bar{m}_t^{n+1}-\bar{m}_t^n|(dz)\leq \frac{C}{n}(1+|x|).
\end{align*}
By using the a priori estimates given in Lemma \ref{lemma a priori estimates}, we deduce
\begin{align*}
-\frac{1}{n+1}\langle f(\bar m^{n+1})-f(\bar m^n), m^{n+1}-\bar m^n \rangle &\leq \frac{1}{n+1}\langle |f(\bar m^{n+1})-f(\bar m^n)|, m^{n+1}+\bar m^n \rangle\\
& \leq \frac{C}{n^2}.
\end{align*}
We obtain
\begin{align*}
&\langle f(\bar m^n), \bar m^n \rangle - \langle f(\bar m^{n+1}), \bar m^{n+1}\rangle = \langle f(\bar m^n), \bar m^n \rangle  \\
&\quad - \langle f(\bar m^{n+1}), \bar m^n + \frac{1}{n+1}(m^{n+1}-\bar m^n) \rangle \\
& = \langle f(\bar m^n) - f(\bar m^{n+1}), \bar m^n \rangle - \frac{1}{n+1}\langle f(\bar m^{n+1}), m^{n+1}-\bar m^n \rangle\\
&\leq \langle f(\bar m^n) - f(\bar m^{n+1}), \bar m^n \rangle - \frac{1}{n+1}\langle f(\bar m^n), m^{n+1}-\bar m^n \rangle + \frac{C}{n^2}.
\end{align*}
Similar to the above inequality, we also derive
\begin{align*}
&\langle g(\bar \mu^n), \bar \mu^n \rangle - \langle g(\bar \mu^{n+1}), \bar \mu^{n+1}\rangle \\
&\leq \langle g(\bar \mu^n) - g(\bar \mu^{n+1}), \bar \mu^n \rangle - \frac{1}{n+1}\langle g(\bar \mu^n), \mu^{n+1}-\bar \mu^n \rangle + \frac{C}{n^2}.
\end{align*}

\noindent Therefore, by combining the last three inequalities, we derive
\begin{align*}
\varepsilon_n^{(1)}  
& = \langle f(\bar m^n), \bar m^n \rangle + \langle g(\bar \mu^n), \bar \mu^n \rangle - \langle f(\bar m^{n+1}), \bar m^{n+1} \rangle - \langle g(\bar \mu^{n+1}), \bar \mu^{n+1} \rangle\\
&+\lambda \mathbb{E}\left[\int_0^T (\mathcal{E}(\bar{\xi_t}^n)-\mathcal{E}({\bar{\xi}_t}^{n+1}))\right]dt\\
&\leq \langle f(\bar m^n) - f(\bar m^{n+1}), \bar m^n \rangle + \langle g(\bar \mu^n) - g(\bar \mu^{n+1}), \bar \mu^n \rangle -\frac{\varepsilon_n}{n+1}+\frac{C}{n^2}.
\end{align*}
We shall now proceed with the estimation of $\varepsilon_n^{(2)}$. First notice that, since
\begin{align}
J^\lambda(\xi^{n+1},\bar{m}^n, \bar{\mu}^n) \geq J^\lambda(\xi,\bar{m}^n, \bar{\mu}^n), \,\, \forall \xi \in \mathcal{A},
\end{align}
it follows that
\begin{align}
& \langle f(\bar{m}^n),m^{n+1} \rangle+ \langle g(\bar{\mu}^n),\mu^{n+1} \rangle +\lambda \mathbb{E}\int_0^T \mathcal{E}(\xi^{n+1}_t)dt  \nonumber \\
& \geq \langle f(\bar{m}^n),m^{n+2} \rangle+ \langle g(\bar{\mu}^n),\mu^{n+2} \rangle +\lambda \mathbb{E}\int_0^T \mathcal{E}(\xi^{n+2}_t)dt.
\end{align}
We therefore have
\begin{align*}
\varepsilon_n^{(2)} & =\langle f(\bar{m}^{n+1}),m^{n+2} \rangle+\langle g(\bar{\mu}^{n+1}),\mu^{n+2} \rangle+\lambda \mathbb{E}\left[\int_0^T (\mathcal{E}(\xi_t^{n+2}) \right] \nonumber \\& -\langle f(\bar{m}^n),m^{n+1} \rangle- \langle g(\bar{\mu}^n),\mu^{n+1} \rangle- \lambda \mathbb{E}\left[\int_0^T\mathcal{E}(\xi_t^{n+1}))dt\right] \\
&\leq \langle f(\bar{m}^{n+1}),m^{n+2} \rangle+\langle g(\bar{\mu}^{n+1}),\mu^{n+2} \rangle -\langle f(\bar{m}^n),m^{n+2} \rangle- \langle g(\bar{\mu}^n),\mu^{n+2} \rangle\\
&=\langle f(\bar m^{n+1}) - f(\bar m^n), m^{n+2}\rangle + \langle g(\bar \mu^{n+1}) - g(\bar \mu^n), \mu^{n+2}\rangle \\
&= \langle f(\bar m^{n+1}) - f(\bar m^n), m^{n+1}\rangle + \langle g(\bar \mu^{n+1}) - g(\bar \mu^n), \mu^{n+1}\rangle \\
&\quad + \langle f(\bar m^{n+1}) - f(\bar m^n), m^{n+2}-m^{n+1}\rangle + \langle g(\bar \mu^{n+1}) - g(\bar \mu^n), \mu^{n+2} - \mu^{n+1}\rangle.
\end{align*}
By Lemma 2.15 in \cite{dumitrescu.leutscher.tankov.2022linear},  the following estimates hold true: there exist constants $C_f>0$ and $C_g>0$ such that for all $n\geq 1$
$$
\langle f(\bar m^{n+1}) - f(\bar m^n),m^{n+2}-m^{n+1}\rangle 
\leq \frac{C_f}{n}d_1^M(m^{n+1}, m^{n+2}),
$$
$$
\langle g(\bar \mu^{n+1}) - g(\bar \mu^n),\mu^{n+2}-\mu^{n+1}\rangle \leq \frac{C_g}{n}d_1(\mu^{n+1}, \mu^{n+2}).
$$
Therefore
\begin{align*}
\varepsilon_n^{(2)} &\leq \langle f(\bar m^{n+1}) - f(\bar m^n), m^{n+1}\rangle + \langle g(\bar \mu^{n+1}) - g(\bar \mu^n), \mu^{n+1}\rangle\\
&\quad +\frac{C}{n}(d_1^M(m^{n+1}, m^{n+2}) + d_1(\mu^{n+1}, \mu^{n+2})).
\end{align*}
We set
$$\delta_n:=C\left[ d_1^M(m^{n+1}, m^{n+2}) + d_1(\mu^{n+1}, \mu^{n+2}) +  \frac{1}{n}\right],$$
and we derive 
\begin{align*}
&\varepsilon_{n+1} - \varepsilon_n = \varepsilon_n^{(1)} + \varepsilon_n^{(2)} \leq \langle f(\bar m^n) - f(\bar m^{n+1}), \bar m^n \rangle + \langle g(\bar \mu^n) - g(\bar \mu^{n+1}), \bar \mu^n \rangle -\frac{\varepsilon_n}{n+1}\\
&\quad + \langle f(\bar m^{n+1}) - f(\bar m^n), m^{n+1} \rangle + \langle g(\bar \mu^{n+1}) - g(\bar \mu^n), \mu^{n+1} \rangle + \frac{\delta_n}{n}\\
& = \langle f(\bar m^{n+1}) - f(\bar m^n), m^{n+1} - \bar m^n\rangle + \langle g(\bar \mu^{n+1}) - g(\bar \mu^n), \mu^{n+1} - \bar \mu^n \rangle -\frac{\varepsilon_n}{n+1} + \frac{\delta_n}{n}+\varepsilon_n^{(3)}\\
& = (N+1)\left[ \langle f(\bar m^{n+1}) - f(\bar m^n), \bar m^{n+1} - \bar m^n\rangle + \langle g(\bar \mu^{n+1}) - g(\bar \mu^n), \bar \mu^{n+1} - \bar \mu^n \rangle  \right] \\
&\quad -\frac{\varepsilon_n}{n+1} + \frac{\delta_n}{n}\\
&\leq -\frac{\varepsilon_n}{n+1} + \frac{\delta_n}{n}.
\end{align*}
The last inequality follows from the Lasry–Lions monotonicity condition. 
By using the same arguments as in Corollary 2.17 in \cite{dumitrescu.leutscher.tankov.2022linear} with the continuous  map $\Gamma \circ R_\lambda$ from $(\mathcal{M}^\star, d_1^M \otimes d_1)$ to $(\mathcal{M}^\star, d_1^M \otimes d_1)$, we derive that
\begin{align}
\lim_{n\rightarrow\infty} d_1(\mu^n,\mu^{n+1})=0,\,\,\lim_{n\rightarrow\infty} d_1^M(m^n,m^{n+1})=0.
\end{align}
Thus $\delta_n\rightarrow 0$. By Lemma 3.1 in \cite{hadikhanloo2017learning}, we conclude that $\varepsilon_n\rightarrow 0$ as $n\rightarrow \infty$. The result follows. 

{\textit{Step 2.} We show now that the sequence $(\bar{\xi}^n, \bar{m}^n, \bar{\mu}^n)$ converges to the unique $\lambda$-SC-MFG equilibrium $(\hat{\xi}^\lambda,\hat{m}^\lambda, \hat{\mu}^\lambda).$ By compactness of the set $\mathcal{A} \times \Gamma(\mathcal{A})$, from any subsequence $(\bar{\xi}^{k_{n}}, \bar{m}^{k_{n}}, \bar{\mu}^{k_{n}})$ we can subtract a further subsequence (not relabeled) such that $(\bar{\xi}^{k_{n}}, \bar{m}^{k_{n}}, \bar{\mu}^{k_{n}})$ converges in the topology $\tau_2^w \otimes d_1^M \otimes d_1$ to  some $(\bar{\xi}, \bar{m},\bar{\mu})$, where $\tau_2^w$ represents the topology associated to the weak convergence in $\mathbb{L}^2(\Omega \times [0,T])$.
First, by using similar arguments as for the proof of \eqref{eq inside exist}, we can establish that
\begin{align}\label{consist}
(\bar{m},\bar{\mu})=\Gamma(\bar{\xi}).
\end{align}
Then, by optimality of $\xi_\cdot^{k_{n+1}}$, we get 
\begin{align}
J^{\lambda}(\xi^{k_{n+1}},\bar{m}^{k_{n}},\bar{\mu}^{k_n}) \geq J^{\lambda}(\xi,\bar{m}^{k_{n}},\bar{\mu}^{k_n}),\,\, \text{for all } \xi \in \mathcal{A}.
\end{align}
By definition of $\varepsilon_{k_n}$, we derive
\begin{align}
J^{\lambda}(\bar{\xi}^{k_{n}},\bar{m}^{k_{n}},\bar{\mu}^{k_n}) \geq J^{\lambda}(\xi,\bar{m}^{k_{n}},\bar{\mu}^{k_n})-\varepsilon_{k_n},\,\, \text{for all } \xi \in \mathcal{A}.
\end{align}
By using similar arguments as the ones used to establish \eqref{eq lemma continuity first limit}, \eqref{eq lemma continuity second limit} and \eqref{eq lemma continuity third limit} and by using the convergence of $(\varepsilon_n)$ to $0$ when $n \to \infty$, we get
\begin{align}
J^{\lambda}(\bar{\xi},\bar{m},\bar{\mu}) \geq J^{\lambda}(\xi,\bar{m},\bar{\mu}),\,\, \text{for all } \xi \in \mathcal{A}.
\end{align}
From the above inequality and relation \eqref{consist}, we conclude that $(\bar{\xi},\bar{m},\bar{\mu})$ is a $\lambda$-SC-MFG equilibrium. By uniqueness of the equilibrium, it coincides with $(\hat{\xi}^\lambda,\hat{m}^\lambda,\hat{\mu}^\lambda)$.
Since from any subsequence of $(\bar{\xi}^n, \bar{m}^n, \bar{\mu}^n)$ we can subtract a further subsequence which converges to the unique $\lambda$-SC-equilibrium, we conclude that the whole sequence converges to the unique MFG equilibrium.}
\end{proof}

\subsection{Fictitious play in the supermodular case} 
We now provide convergence of the fictitious play algorithm when the so-called supermodularity condition holds true (see, e.g., \cite{dianetti2022strong, dianetti.ferrari.fischer.nendel.2022unifying, vives.vravosinos.2024}).  
Such a property naturally appears in many financial and economic applications, as for example in models for bank runs (see \cite{carmona.delarue.lacker.2017.timing} and Example \ref{example on bank runs} below). 
The interested reader is referred to the textbook \cite{Vives01}  for further details on supermodular games.

\noindent We thus enforce the following structural condition:
\begin{assumption}[Supermodularity]\label{assumption supermod} 
The following hold true:
\begin{enumerate}
\item\label{assumption supermod: uniqueness R}
For any $(m,\mu)\in \mathcal{M}$, $R_\lambda (m,\mu)$ is a singleton.
\item\label{assumption supermod: supermod}  For each \( (t, x) \in [0, T] \times \mathbb{R}^d \), the function \( f(t, x, \cdot) \) is nondecreasing with respect to the measure argument in the following sense: for any two measures \( m, \bar{m} \in \mathcal P _p^{sub} (\R^d) \), if \( m(A) \leq \bar{m}(A) \) for all \( A \in \mathcal{B}(\mathbb{R}^d) \), then
\[
f(t, x, m) \leq f(t, x, \bar{m}).
\]  
\item\label{assumption supermod: supermod g} We have $g(x,\mu)= \tilde g (x , \langle \psi, \mu\rangle )$, for some functions $\tilde g \in \mathcal C^2(\R^d \times \R)$ and $\psi \in \mathcal C^{1,2} ([0,T] \times \R^d)$ such that $(\partial _t + \mathcal L)  \psi \geq 0$ and $y \mapsto \mathcal L  \tilde g (x,y)$ is nondecreasing for any $x \in \R^d$.
\end{enumerate}
\end{assumption}
For the sake of illustration, we discuss some key examples.
\begin{example}
\label{example on bank runs}
An example in which the supermodularity condition is satisfied is when the function $g$ does not depend on $\mu$ (in this case, Condition \eqref{assumption supermod: supermod g} above is not needed) and 
$f (t,x,m)= f_0(t,x,1-m(\R^d))$, with $f_0(t,x,\cdot)$ nonincreasing for any $t,x$.
This setting represents supermodular MFGs with equilibrium condition
$$
\hat \tau \in \argmax_{\tau} \E \bigg[ \int_0^\tau f_0(t,X_t, \hat m_t) dt + g(X_\tau) \bigg]
\and
\hat m_t = \P [\hat \tau \leq t], 
$$
which are also referred to as MFGs of timing. 
The supermodularity condition is typical of models for bank runs \cite{carmona.delarue.lacker.2017.timing}.
Natural examples are of additive type $f_0 (t,x,y)= f_1(x) - f_2(y)$ or of multiplicative type $f_0 (t,x,y)= - f_1(x) f_2(y)$ (with $f_1 \geq 0$), for nondecreasing functions $f_1,f_2$. 
\end{example}

\begin{example}
The underlying process evolves as a one dimensional geometric Brownian motion
$$
dZ_t = Z_t (b_0 dt + \sigma_0 d W_t), \quad Z_0 = z_0 > 0, \quad \P \text{-a.s.,}
$$
with $b_0\in \R$, $\sigma_0>0$, and the payoffs are
$$
f(t,z,m) = \int_{\R} (z+y) m(dy)
\and
g(t,z,\mu) = \int_{\R \times [0,T]} (t+s) \mu(ds, dy).
$$
With respect to the previous notation, the state process becomes $X_t:=(t,Z_t)$ so that time-dependence in $g$ is allowed. 
\end{example}

Before discussing the main result of this subsection, we recall that the controls $\xi$ are interpreted as probability measures; in particular, $\xi_t$ is the probability of stopping before time $t$.
In this sense, if $\xi_t \geq \bar \xi_t$ for any $t\in [0,T]$, $\mathbb P$-a.s., then $\tau^\xi \leq \tau ^{\bar \xi}$ $\mathbb P$-a.s., and we say that $\xi$ is earlier than $\bar \xi$ (or that $\bar \xi$ is later than $\xi)$ .
Given two equilibria $(\xi, m, \mu)$ and $(\bar \xi, \bar m, \bar \mu)$, we say that the first equilibrium is earlier (resp.\ later) that the second one if $\xi$ is earlier (resp.\ later) that $\bar \xi$.
Given $\lambda \geq 0$, the earliest equilibrium   $(\underline \xi^\lambda, \underline m^\lambda, \underline \mu^\lambda)$ and the latest equilibrium 
$(\overline \xi ^\lambda, \overline m^\lambda, \overline \mu^\lambda)$
are such that 
$\underline \xi^\lambda_t \geq \xi_t \geq \overline \xi^\lambda_t$ for any $t\in [0,T]$, $\mathbb P$-a.s.,
for any other equilibrium $(\xi, m, \mu)$.

The following theorem discusses the existence and approximation of the earliest and latest equilibria. 
 
\begin{theorem}\label{theorem fictitious supermodular}
Under Assumptions \ref{assumption} and \ref{assumption supermod}, the earliest equilibrium   $(\underline \xi^\lambda, \underline m^\lambda, \underline \mu^\lambda)$ and the latest equilibrium $(\overline \xi^\lambda, \overline m^\lambda, \overline \mu^\lambda)$ exist. 
Moreover, the following statements hold true: 
\begin{enumerate}
    \item With initialization $\bar \xi^0_t =1$ for any $t \in [0,T]$, the sequence $\bar \xi^n$ is nonincreasing, $\bar \xi^{n+1} \leq \bar \xi ^n$, and $(\bar \xi^n, \bar m^n, \bar \mu^n)$   converges to the earliest equilibrium   $(\underline \xi^\lambda, \underline m^\lambda, \underline \mu^\lambda)$.
    \item With initialization $\bar \xi^0_t  =0$ for any $t \in [0,T]$, the sequence $\bar \xi^n$ is nondecreasing, $\bar \xi^{n+1} \geq \bar \xi ^n$, and $(\bar \xi^n, \bar m^n, \bar \mu^n)$ converges to the latest equilibrium $(\overline \xi^\lambda, \overline m^\lambda, \overline \mu^\lambda)$.
\end{enumerate}
\end{theorem}

\begin{proof} 
We limit ourself to show the claim on the existence and approximation of the earliest equilibrium, as the proof for the latest equilibrium follows by the same argument.  
The proof is divided into three steps.
\smallbreak\noindent
\emph{Step 1.} 
In this step we show the some monotonicity properties of the maps $\Gamma$ and $R _\lambda$.

We first show anti-monotonicity property of the consistency map $\Gamma$.
Observe that,
if $\xi_t \leq \xi'_t, \, \forall t \in [0,T], \, \mathbb P$-a.s., from the definition of 
$m^\xi, m^{\xi'}$
have 
\begin{equation*}
     m^\xi_t (A) = \E [ \mathds{1}_A(X_t) (1-\xi_t) ] \geq \E [ \mathds{1}_A (X_t) (1- \xi'_t) ] =  m^{\xi'}_t (A), 
      \quad  \forall A \in \mathcal B (\R^d), \ t \in [0,T].
\end{equation*}
Moreover, for $\psi$ as in Assumption \ref{assumption supermod}, by using integration by parts and  that $\xi_T =1$, we have
$$
\langle \psi, \mu^{\xi} \rangle  :=  \E\bigg[\int_{[0,T]} \psi(t,X_t) d\xi_t\bigg] = \E \bigg[ \psi (T,X_T) - \int_0^T \xi_t ( \partial _t + \mathcal L ) \psi (t, X_t) dt \bigg],
$$
and the analogous expression can be obtained for  $\xi'$. 
Thus, since $ (\partial _t + \mathcal L )\psi \geq 0$, we find
\begin{align*}
    \langle \psi, \mu^{\xi} \rangle  &= \E \bigg[ \psi (T,X_T) - \int_0^T \xi_t ( \partial _t + \mathcal L )   \psi (t, X_t) dt \bigg] \\
    & \geq \E \bigg[ \psi (T,X_T) - \int_0^T \xi' _t ( \partial _t + \mathcal L ) \psi (t, X_t) dt \bigg] = \langle \psi, \mu^{\xi'} \rangle . 
\end{align*}
To summarize, we have found that 
\begin{equation}\label{eq projection decreasing}
\text{if $\xi_t \leq \xi'_t, \, \forall t, \, \mathbb P$-a.s., \ \ then $ m^\xi_t (A)  \geq   m^{\xi'}_t (A), 
      \,  \forall A, t,$ and $\langle \psi, \mu^{\xi} \rangle   \geq 
      \langle \psi, \mu^{\xi'} \rangle$,}
\end{equation}
which is the desired anti-monotonicity of $\Gamma$.

We next show the anti-monotonicity of the best reply map $R _\lambda$; that is, 
for any $m, m'$, 
\begin{equation}
    \label{eq monotonicity best reply}
    \text{if $ m_t (A)  \geq   m'_t (A), 
      \,  \forall A, t,$ and $\langle \psi, \mu  \rangle   \geq 
      \langle \psi, \mu' \rangle$, \ \ then $ R _\lambda (m, \mu)_t \leq  R _\lambda (m ', \mu')_t, \, \forall t, \, \mathbb P$-a.s. }
\end{equation}
 To see this, set $\xi : =  R _\lambda (m, \mu), \,  \xi':=  R _\lambda (m ' , \mu')$,  and define the processes
 $$
 \xi \land \xi':= \big( \min\{ \xi_t, \xi'_t \} \big)_t
 \and
 \xi \lor \xi':= \big( \max\{ \xi_t, \xi'_t \} \big)_t.
 $$
For a generic process $\zeta \in \mathcal A$, using integration by parts for the integral in $d \zeta$, we find
$$
J^{\lambda}(\zeta, m,\mu)  = \E \bigg[ - \int _0^T \big( f(t,X_t,m_t) + \mathcal L g (X_t,\mu) \big) \zeta_t dt + \int_0^T \big( f(t,X_t,m_t) +\lambda \mathcal E (\zeta_t) \big)  dt + g(X_T, \mu) \bigg],
$$
and, defining $\hat f (t,x,m,\mu) =  f(t,x,m) + \mathcal L g (x,\mu) $, we rewrite the latter expression as
$$
J^{\lambda}(\zeta, m,\mu)  = \E \bigg[ - \int _0^T \hat f(t,X_t,m_t,\mu) \zeta_t dt + \int_0^T \big( f(t,X_t,m_t) +\lambda \mathcal E (\zeta_t) \big)  dt + g(X_T, \mu) \bigg].
$$
Now, noticing that $\xi'_t -\xi_{t} \land \xi'_t = \xi_t \lor \xi'_t-\xi_{t}$ and that $\mathcal E ( \xi_{t} \lor \xi'_t ) - \mathcal E ( \xi'_t) = \mathcal E (\xi_{t} ) -\mathcal E ( \xi_t \land \xi'_t )$, we first find
$$
\begin{aligned}
&  J^{\lambda}(\xi \lor \xi', m',\mu')  - J^{\lambda}(\xi', m',\mu')  \\
&\quad
=  \mathbb{E}\left[\int_{0}^{T}  \big(  \hat f (t,X_{t}, m'_{t}, \mu ' ) \left( \xi'_t  - \xi_{t} \lor \xi'_t \right) 
+ \lambda ( \mathcal E ( \xi_{t} \lor \xi'_t ) - \mathcal E ( \xi'_t) ) \big)  dt \right] \\
&\quad
=  \mathbb{E}\left[\int_{0}^{T}  \big( \hat f (t,X_{t}, m'_{t}, \mu ' ) \left( \xi_t \land \xi'_t-\xi_{t} \right) 
+ \lambda ( \mathcal E (\xi_{t} ) -\mathcal E ( \xi_t \land \xi'_t ) ) \big)  dt \right] .
\end{aligned}
$$
Thus, if $ m_t (A)  \geq   m'_t (A), 
      \,  \forall A, t,$ and $\langle \psi, \mu  \rangle   \geq 
      \langle \psi, \mu' \rangle$,
from Conditions \ref{assumption supermod: supermod} and \ref{assumption supermod: supermod g} in Assumption \ref{assumption supermod} we have
$$
\begin{aligned}
&  J^{\lambda}(\xi \lor \xi', m', \mu')  - J^{\lambda}(\xi', m', \mu')  \\
  &\quad 
  \geq \mathbb{E}\left[\int_{0}^{T} \big( \hat f (t,X_{t}, m_{t}, \mu_t ) \left( \xi'_t \land \xi_t-\xi_{t}\right) + \lambda ( ( \mathcal E (\xi_{t} ) -\mathcal E ( \xi_t \land \xi'_t ) ) \big) dt \right] \\
& \quad 
= J^{\lambda}(\xi, m, \mu)  - J^{\lambda}(\xi \land \xi ', m, \mu).
\end{aligned}
$$
Hence, from the optimality of $\xi'$ for $(m',\mu')$ and the fact that $\xi \lor \xi' \in \A$, we deduce that
$$
0\geq  J^{\lambda}(\xi \lor \xi', m', \mu')  - J^{\lambda}(\xi', m', \mu')
\geq J^{\lambda}(\xi, m,\mu)  - J^{\lambda}(\xi \land \xi ', m,\mu).
$$
This  in turn implies that $\xi \land \xi'
= R_\lambda (m,\mu)$, and, by uniqueness of the optimizer as in Condition \ref{assumption supermod: uniqueness R} in Assumptoin \ref{assumption supermod}, we conclude that $\xi \land \xi' = \xi$, thus proving \eqref{eq monotonicity best reply}.

\smallbreak\noindent
\emph{Step 2.} 
In this step we show the monotonicity of $\bar \xi ^n$ by an induction argument.
To simplify the notation, we will write $m \geq m'$ instead of $ m_t (A)  \geq   m'_t (A), 
      \,  \forall A, t$.
Moreover, 
$(m^k,\mu^k):= \Gamma(\xi^k)$, so that (due to the linear dependence of the measure $(m,\mu)$ with respect to $\xi$) we have the relation 
$$
(\bar{\mu}^{n}, \bar{m}^{n})=\frac{1}{n}  \sum_{k=1}^{n}({\mu}^{\xi^k}, {m}^{\xi^k}) = \frac{1}{n} \sum_{k=1}^{n}({\mu}^{k},  {m}^{k}),
$$
that will be used several times in the sequel.

We proceed with an induction argument.
Since  $\bar \xi ^0 \equiv 1$ and $(\bar m ^0, \bar \mu ^0) =( m^{\xi^0}, \mu^{\xi^0})$, we have
$$   
\begin{aligned} 
\xi^1 &= R_\lambda (\bar m ^0, \bar \mu ^0) \leq \bar \xi^0, \\
(m^1,\mu^1) &= (m^{\xi^1},\mu^{\xi^1}),
\\
\bar m ^1 &= m^1 = m^{\xi^1} \geq m^{ \bar \xi^0} =  \bar m ^0, \\ 
\langle \psi, \bar \mu ^1 \rangle &=  \langle \psi, \mu^1 \rangle =   \langle \psi, \mu^{\xi^1} \rangle \geq \langle \psi, \mu^{\bar \xi^0} \rangle = \langle \psi, \bar \mu ^0 \rangle.
\end{aligned}
$$
Hence, by the monotonicity of $R_\lambda$ in \eqref{eq monotonicity best reply} and of $\Gamma$ in \eqref{eq projection decreasing}, we have
$$
\begin{aligned}
    \xi^2 &= R_\lambda (\bar m ^1, \bar \mu ^1) \leq R_\lambda (\bar m ^0, \bar \mu ^0) = \xi^1, \\
m^2 &= m^{\xi^2} \geq m^{\xi^1} = m^1, \\
\langle \psi,  \mu ^2 \rangle &=  \langle \psi, \mu^{\xi^2} \rangle \geq  \langle \psi, \mu^{\xi^1} \rangle = \langle \psi,  \mu ^1 \rangle , \\
\bar m ^2 &= \frac12 (m^2 +  m^1) \geq m^1 = \bar m ^1, \\ 
\langle \psi, \bar \mu ^2 \rangle &=  \langle \psi, \frac12 (\mu^2 +  \mu^1) \rangle 
\geq \langle \psi,  \mu ^1 \rangle
\geq \langle \psi, \bar \mu ^0 \rangle.
\end{aligned}
$$
Assume that for a generic $n$ we have
$$
\begin{aligned}
\xi^n & \leq \xi^{n-1} \leq \dots \leq \xi^1, \\
m^n & =  m^{\xi^n} \geq m^{n-1} \geq \dots \geq  m^1, \\
\langle \psi,  \mu ^n \rangle &=  \langle \psi, \mu^{\xi^n} \rangle \geq  \langle \psi, \mu^{n-1} \rangle \geq \dots \geq  \langle \psi,  \mu ^1 \rangle , \\
\bar m ^n &=\frac{1}{n} \sum_{k=1}^n m^k \geq \bar m ^{n-1} \geq \dots \geq \bar m ^1, \\ 
\langle \psi, \bar \mu ^n \rangle &=  \langle \psi, \frac{1}{n} \sum_{k=1}^n \mu^k \rangle 
\geq \langle \psi,  \bar \mu ^{n-1} \rangle
\geq \dots \geq  \langle \psi, \bar \mu ^1 \rangle.
\end{aligned}
$$
Then, the monotonicity of $R_\lambda$ in \eqref{eq monotonicity best reply} and of $\Gamma$ in \eqref{eq projection decreasing},  again gives
$$
\begin{aligned}
\xi^{n+1} &= R_\lambda (\bar m ^n, \bar \mu ^n) \leq R_\lambda (\bar m ^{n-1}, \bar \mu ^{n-1}) = \xi^n, \\
m^{n+1} &=  m^{\xi^{n+1}} \geq m^{\xi^{n}} = m^{n} \geq \dots \geq m^1, \\
\langle \psi,  \mu ^{n+1} \rangle &=  \langle \psi, \mu^{\xi^{n+1}} \rangle \geq  \langle \psi, \mu^{\xi^{n}} \rangle = \langle \psi, \mu^{{n}} \rangle \geq \dots \geq \langle \psi,  \mu ^1 \rangle .
\end{aligned}
$$
Moreover, from the last two inequalities we deduce that
$$
\begin{aligned}
m^{n+1} \geq \bar m ^n =\frac{1}{n} \sum_{k=1}^n m^k
\and
\langle \psi,  \mu ^{n+1} \rangle \geq 
\langle \psi, \bar \mu ^{n} \rangle = \langle \psi, \frac{1}{n} \sum_{k=1}^n \mu^k \rangle,
\end{aligned}
$$
which allows to conclude that
$$
\begin{aligned}
\bar m ^{n+1} & = \frac1{n+1} (m^{n+1} + n \, \bar m ^n) \geq \bar m ^n, \\
\langle \psi, \bar \mu ^{n+1} \rangle & = \langle \psi, \frac1{n+1} (\mu^{n+1} + n \, \bar \mu ^n) \rangle \geq \langle \psi, \bar \mu ^{n} \rangle,
\end{aligned}
$$
thus completing the induction argument. 

Finally, the monotonicity of the sequence $\bar \xi^n$ follows from the monotonicity of $\xi ^n$. 

\smallbreak\noindent
\emph{Step 3.} 
 We next define (thanks to the monotonicity of $\xi^n$ and $\bar \xi^n$) the limit points
$$ 
\underline \xi := \inf_n \xi^n  = \inf_n \bar \xi^n, \quad 
\underline m := m^{\underline \xi},
\and
\underline \mu := \mu^{\underline \xi},
$$ 
and show that these coincide with the minimal equilibria $(\underline \xi ^\lambda , \underline m ^ \lambda, \underline \mu ^ \lambda )$.

Notice that the sequence $(\bar \xi ^n)_n$ converges strongly. Moreover, by monotonicity, it is easy to see that the consistency map $\Gamma $ is continuous along this sequence.
Therefore,  we have
$$
(\underline m, \underline \mu) = (m^{\underline \xi}, \mu^{\underline \xi} ) = \Gamma (\underline \xi) = \lim_n \Gamma (\bar \xi^n) = \lim_n (\bar m^n, \bar \mu^n),
$$
so that the sequence $(\bar m ^n, \bar \mu ^n)_n$ converges to $ (\underline m, \underline \mu)$ as well.
Hence, the continuity of $R_\lambda$ in turn implies that
$$
\underline \xi  = \lim _n R_\lambda (\bar m ^n, \bar \mu ^n) = R_\lambda (\underline m, \underline \mu), 
$$
which proves the optimality. 
This show that $(\underline \xi , \underline m, \underline \mu)$ is a mean-field equilibrium. 

Finally, we show the minimality of $(\underline \xi , \underline m, \underline \mu)$.
Let $(\hat \xi, \hat m, \hat \mu)$ be another mean-field equilibrium.
By definition of $\bar \xi ^0$, we have $\hat \xi \leq \bar \xi ^ 0$.
Hence, by monotonicity of $\Gamma$ in \eqref{eq projection decreasing}, we have $m^{\hat \xi} = \hat m \geq m^1 = m^{\xi^1} = \bar m^1$ and 
$\langle \psi, \mu^{\hat \xi} \rangle = \langle \psi,  \hat \mu \rangle \geq \langle \psi,  \mu^1 \rangle = \langle \psi,  \mu^{\xi^1}\rangle = \langle \psi, \bar \mu^1\rangle$. 
Hence, using the monotonicity of $R_\lambda$ so that $\hat \xi = R_\lambda (\hat m, \hat \mu) \leq R_\lambda (\bar m^1, \bar \mu^1) = \xi^2$ by monotonicity of $R_\lambda $.
Proceeding by induction, we conclude that
$$
\hat \xi \leq \bar \xi^n, \quad \forall n,
$$
so that $\hat \xi \leq \underline \xi =\inf_n \bar \xi^n$,
thus giving the desired minimality.
\end{proof}

\appendix

\section{Connection between the $\lambda$-singular control MFG problem and a $\lambda$-optimal stopping MFG problem}
\label{sec:App}




In this part, we present the connection between the singular mean-field problem $\lambda$-SC-MFG and a {\it new} type of OS-MFG problems, that we introduce below and we call $\lambda$-OS-MFG ($\lambda$-optimal stopping MFG). Theorem \ref{first_implic} below paves the way to the characterization of the optimal free boundary of the entropy regularized problem $\lambda$-SC-MFG, as it allows via probabilistic techniques to derive an integral equation uniquely satisfied by the boundary. This perspective is also insightful for learning algorithm development. A detailed analysis can be found in the companion paper \cite{jodietal.}.

Let $\bar{\mathcal{T}}$ denote the set of $\mathbb{F}^{x_0, W,U}$-stopping times with values in $[0,T]$, where the filtration  $\mathbb{F}^{x_0, W,U}$ has been introduced in Section \ref{secreg}. For ease of exposition, we restrict ourselves to the relative entropy $\mathcal E (z) = -z\log z$; the analysis for a general entropy function satisfying Assumption \ref{assumption entropy} is analogous.

\begin{definition}
\label{def:defOSequilibrium}
A triple $(\hat{m}^\lambda, \hat{\mu}^\lambda, \hat{\theta}^\lambda)$ is said to be an equilibrium to the $\lambda$-OS-MFG problem if:
\begin{align}
\begin{cases}
&\hat{\theta}^\lambda \in \argmax _{\theta \in \bar{\mathcal{T}}} \mathbb{E}\left[\int_{0}^{\theta}\left(f (t, X_{t}, \hat{m}^\lambda_t) -\lambda(1+\log U)\right) d t+g (X_{\theta}, \hat{\mu}^\lambda )\right] \nonumber \\
& \hat{m}^\lambda_{t}(A)=\mathbb{P}\left[X_{t} \in A, t \leq \hat{\theta}^\lambda\right],\,\, A \in \mathcal{B}(\mathbb{R}^d) \nonumber \\
& \hat{\mu}^\lambda(B \times[0, t])=\mathbb{P}\left[X_{\hat{\theta}} \in B, \hat{\theta}^\lambda \leq t\right]\,\, B \in \mathcal{B}(\mathbb{R}^d). \nonumber\\
\end{cases}
\end{align}
\end{definition}


Observe that, for any stopping time $\theta \in \bar{\mathcal{T}}$, since  $\theta$ is a $\mathbb{F}^{x_0,W,U}$-stopping time, then there exists a measurable function, still denoted by $\theta$,  such that $\theta=\theta(x_0, W,U)$. 

We now establish the first relation between the two game problems.
\begin{theorem}\label{first_implic}
     Let $(\hat{m}^\lambda, \hat{\mu}^\lambda, \hat{\xi}^\lambda)$ be an equilibrium to the $\lambda$-SC-MFG.
     Define
     \begin{align}\label{def0}
\hat{\theta}^\lambda:=\theta^{\lambda}(\cdot,1-U),
     \end{align}
     where 
     \begin{align}
\theta^\lambda(\cdot,u):=\inf \{t \geq 0| \hat{\xi}^\lambda_t \geq u\}.
     \end{align}
\noindent Then $(\hat{m}^\lambda, \hat{\mu}^\lambda, \hat{\theta}^\lambda)$ is an equilibrium to the $\lambda$-OS-MFG .
\end{theorem}
\begin{proof}

Let $(\hat{m}^\lambda, \hat{\mu}^\lambda, \hat{\xi}^\lambda)$ be an $\lambda$-SC-MFG equilibrium and let $\hat{\theta}^\lambda$ be given by $\eqref{def0}$. Observe that 
\begin{align}\label{eqq}
\{t < \hat{\theta}^\lambda\}=\{\hat{\xi}^\lambda_t < 1-U\}.
\end{align}
In view of this observation, we obtain:
\begin{align*}
&\mathbb{E}\left[\int_0^{\hat{\theta}^\lambda}f(t,X_t,\hat{m}^\lambda_t)dt\right]=\mathbb{E}\left[\int_0^T f(t,X_t,\hat{m}^\lambda_t)\mathds{1}_{\{t<\hat{\theta}^\lambda\}}dt\right]=\mathbb{E}\left[\int_0^T f(t,X_t,\hat{m}^\lambda_t)\mathbb{P}(U < 1-\hat{\xi}^\lambda_t|\mathcal{F}_t^{x_0,W})dt\right] \nonumber \\
&= \mathbb{E}\left[\int_0^T f(t,X_t,\hat{m}^\lambda_t)(1-\hat{\xi}^\lambda_t)dt\right]\end{align*}
Compute now
\begin{align*}
&\mathbb{E}\left[\int_0^{\hat{\theta}^\lambda}(1+\log(U))dt \right]=\mathbb{E}\left[\int_0^T \mathbb{E}\left[ (1+\log(U))\mathds{1}_{\{U < 1-\hat{\xi}^\lambda_t\}}|\mathcal{F}_t^{x_0,W} \right]dt\right] \nonumber \\
& = \mathbb{E} \left[\int_0^T \left(\int_{0}^{1-\hat{\xi}^\lambda_t}(1+\log(u))du\right) dt \right]=\mathbb{E}\left[\int_0^T (1-\hat{\xi}^\lambda_t)\log(1-\hat{\xi}^\lambda_t)dt\right].
\end{align*}
We get
\begin{align}\label{p4}
&\mathbb{E}\left[\int_0^{\hat{\theta}^\lambda}f(t,X_t,\hat{m}^\lambda_t)dt\right]-\lambda \mathbb{E}\left[\int_0^{\hat{\theta}^\lambda}(1+\log(U))dt \right]=\nonumber \\
&\mathbb{E}\left[\int_0^T f(t,X_t,\hat{m}^\lambda_t)(1-\hat{\xi}^\lambda_t)dt\right]-\lambda \mathbb{E}\left[\int_0^T (1-\hat{\xi}^\lambda_t)\log(1-\hat{\xi}^\lambda_t)dt\right].
\end{align}
By applying the "change of variable" formula from Chapter 0, Proposition 4.9 in \cite{revuz.yor.2013continuous}, we also have
\begin{align}\label{p5}
\mathbb{E}[g(X_{\hat{\theta}^\lambda},\hat{\mu}^\lambda)]=\mathbb{E}\left[\mathbb{E}[g(X_{\hat{\theta}^\lambda},\hat{\mu}^\lambda)|\mathcal{F}_T^{x_0,W}]\right]=\mathbb{E}\left[\int_0^1 g(X_{\hat{\theta}^\lambda(\cdot,u)},\hat{\mu}^\lambda)du\right] =\mathbb{E}\left[\int_0^T g(X_t,\hat{\mu}^\lambda)d\hat{\xi}^\lambda\right].
\end{align}
Thus, in view of \eqref{p4} and \eqref{p5}, we derive:
\begin{align}
&\mathbb{E}\left[\int_0^{\hat{\theta}^\lambda}f(t,X_t,\hat{m}^\lambda_t)dt\right]-\lambda \mathbb{E}\left[\int_0^{\hat{\theta}^\lambda}(1+\log(U))dt \right]+\mathbb{E}[g(X_{\hat{\theta}^\lambda},\hat{\mu}^\lambda)]=\nonumber \\
&\mathbb{E}\left[\int_0^T f(t,X_t,\hat{m}^\lambda_t)(1-\hat{\xi}^\lambda_t)dt\right]-\lambda \mathbb{E}\left[\int_0^T (1-\hat{\xi}^\lambda_t)\log(1-\hat{\xi}^\lambda_t)dt\right]+\mathbb{E}\left[\int_0^T g(X_t,\hat{\mu}^\lambda)d\hat{\xi}^\lambda\right].
\end{align}
It remains to show that 
\begin{align}\label{res}
&\mathbb{E}\left[\int_0^{\hat{\theta}^\lambda}f(t,X_t,\hat{m}^\lambda_t)dt\right]-\lambda \mathbb{E}\left[\int_0^{\hat{\theta}^\lambda}(1+\log(U))dt \right] +\mathbb{E}[g(X_{\hat{\theta}^\lambda},\hat{\mu}^\lambda)] \geq \nonumber \\
& \mathbb{E}\left[\int_0^{\theta}f(t,X_t,\hat{m}^\lambda_t)dt\right]-\lambda \mathbb{E}\left[\int_0^{\theta}(1+\log(U))dt \right]+\mathbb{E}[g(X_{{\theta}},\hat{\mu}^\lambda)], \text{ for all } \theta \in \bar{\mathcal{T}}.
\end{align}

We assume that $\hat{\theta}^\lambda \in \bar{\mathcal{T}}$ is not a maximiser and proceed by contradiction. Let   $\bar{\theta} \in \bar{\mathcal{T}}$ be the smallest optimal stopping time. By assumption, we have 
\begin{align}\label{p2}
&\mathbb{E}\left[\int_0^{\bar{\theta}}f(t,X_t,m_t)dt\right]-\lambda \mathbb{E}\left[\int_0^{\bar{\theta}}(1+\log(U))dt \right] +\mathbb{E}[g(X_{\bar{\theta}},{\mu}^\lambda)]  \nonumber \\
& > \mathbb{E}\left[\int_0^{\hat{\theta}^\lambda}f(t,X_t,m_t)dt\right]-\lambda \mathbb{E}\left[\int_0^{\hat{\theta}^\lambda}(1+\log(U))dt \right]+\mathbb{E}[g(X_{\hat{\theta}^\lambda},\hat{\mu}^\lambda)].
\end{align}
Since $\bar{\theta}=\bar{\theta}(x_0,W,U)$ is the smallest optimal stopping time, it follows by similar arguments as in  Lemma 1 in \cite{baldursson1996irreversible} that the map $u \to \bar{\theta}(x_0,W,u)$ is a.s. decreasing, therefore we can define as in \eqref{def0} the right-continuous generalized inverse $\bar{\xi}_t$ as follows:
\begin{align*}
     \bar{\xi}_t:=\sup\{u \in [0,1]: \bar{\theta}(x_0,W,1-u)\leq t\}.\end{align*}
Observe that we have
\begin{align*}
&\mathbb{E}\left[\int_{0}^{\bar{\theta}}\left(f (t, X_{t}, \hat{m}^\lambda_t) -\lambda(1+\log U)\right) d t+g (X_{\bar{\theta}}, \hat{\mu}^\lambda )\right] \nonumber \\
& = \mathbb{E}\left[\int_{0}^T\left(f (t, X_{t}, \hat{m}^\lambda_t) -\lambda(1+\log U)\right)\mathds{1}_{\{t < \bar{\theta}\}} d t+g (X_{\bar{\theta}}, \hat{\mu}^\lambda ) \right]
\end{align*}
By definition of $(\bar{\xi}_t)$, it holds that
\begin{align}
\{ \bar{\theta}(\cdot,u) \leq t\}=\{ \bar{\xi}_t \geq 1-u\}.
\end{align}

In view of the above, and then by conditioning with respect to $\mathbb{F}^{x_0,W}$, using the measurability of $f(t,X_t, \hat{m}_t^\lambda)$ with respect to $\mathbb{F}^{x_0,W}$ and again the "change of variable" formula, we get
\begin{align}\label{p1}
& \mathbb{E}\left[\int_{0}^T\left(f (t, X_{t}, \hat{m}^\lambda_t) -\lambda(1+\log U)\right)\mathds{1}_{\{t < \bar{\theta}\}} d t+g (X_{\bar{\theta}}, \hat{\mu}^\lambda ) \right] \nonumber \\
& = \mathbb{E}\left[\int_{0}^T\left(f (t, X_{t}, \hat{m}^\lambda_t)(1-\bar{\xi}_t)dt\right] -\mathbb{E}\left[ \int_0^T \left(\int_0^1 \lambda(1+\log u)\right)\mathds{1}_{\{t < \bar{\theta}(x_0,W,u)\}}\right)du d t+g (X_{\bar{\theta}}, \hat{\mu}^\lambda )\right] \nonumber \\
& = \mathbb{E}\left[\int_{0}^T\left(f (t, X_{t}, \hat{m}^\lambda_t)(1-\bar{\xi}_t)dt\right] -\mathbb{E}\left[ \int_0^T \left(\int_0^1 \lambda(1+\log u)\right)\mathds{1}_{\{1-u > \bar{\xi_t}\}}\right)du d t+\int_0^1 g (X_{\bar{\theta}(W,u)}, \hat{\mu}^\lambda )du \right] \nonumber \\
& \leq  \mathbb{E}\left[\int_{0}^T(f (t, X_{t}, \hat{m}^\lambda_t)(1-\bar{\xi}_t)dt\right] -\mathbb{E}\left[ \int_0^T  \lambda(1-\bar{\xi}_t)\log (1-\bar{\xi}_t) d t\right]+\mathbb{E}\left[\int_0^T  g (X_{t}, \hat{\mu}^\lambda )d\bar{\xi}_t\right].
\end{align}
From \eqref{p4}-\eqref{p2}-\eqref{p1}, it follows that 
\begin{align}
&\mathbb{E}\left[\int_{0}^T(f (t, X_{t}, \hat{m}^\lambda_t)(1-\bar{\xi}_t)dt\right] -\mathbb{E}\left[ \int_0^T  \lambda(1-\bar{\xi}_t)\log (1-\bar{\xi}_t) d t\right]+\mathbb{E}\left[\int_0^T  g (X_{t}, \hat{\mu}^\lambda )d\bar{\xi}_t\right] \nonumber \\
&>\mathbb{E}\left[\int_0^T f(t,X_t,m_t)(1-\hat{\xi}^\lambda_t)dt\right]-\lambda \mathbb{E}\left[\int_0^T (1-\hat{\xi}^\lambda_t)\log(1-\hat{\xi}^\lambda_t)dt\right]+\mathbb{E}\left[\int_0^T  g (X_{t}, \hat{\mu}^\lambda )d\hat{\xi}^\lambda_t\right],
\end{align}
which contradicts the optimality of $\hat{\xi}^\lambda$. Therefore \eqref{res} holds true.

By using the same arguments as in the derivation of the consistency conditions \eqref{eq consistency m singular control} and \eqref{eq consistency mu singular control}, we also derive that 
$$\hat{m}_t^\lambda(A)=\mathbb{P}\left[X_t \in A, t \leq \hat{\theta}^\lambda\right],\, A \in \mathcal{B}(\mathbb{R}^d)$$
and
$$\hat{\mu}^\lambda(B \times [0,t])=\mathbb{P}\left[X_{\hat{\theta}^\lambda} \in B, \hat{\theta}^\lambda \leq t\right],\,\, B \in \mathcal{B}(\mathbb{R}^d).$$
\end{proof}

We now present the converse result.
\begin{theorem}
     Let $(\hat{m}^\lambda, \hat{\mu}^\lambda, \hat{\theta}^\lambda)$ be an equilibrium to the $\lambda$-OS-MFG. Assume that $u \to \hat{\theta}^\lambda(\cdot,u)$ is a.s.\ decreasing. Define
     \begin{align}\label{def}
     \hat{\xi}_t^\lambda:=\sup\{u \in [0,1]: \hat{\theta}^\lambda(x_0,W,1-u)\leq t\}.\end{align}
     Then $(\hat{m}^\lambda, \hat{\mu}^\lambda, \hat{\xi}^\lambda)$ is an equilibrium to $\lambda$-SC-MFG .
    \end{theorem} 
\begin{proof}
The proof follows the same arguments as for the previous result, therefore we omit it.
\end{proof}

Notice that the assumed requirement that $u \to \hat{\theta}^\lambda(\cdot,u)$ is a.s.\ decreasing can be easily shown to be valid when $\hat{\theta}^\lambda(\cdot,u)$ is the smallest optimal stopping time for the optimal stopping problem of Definition \ref{def:defOSequilibrium} (as it was already used in the proof of Theorem \ref{first_implic}).

\medskip 
\noindent \textbf{Acknowledgements.} 
Supported  by the Deutsche Forschungsgemeinschaft (DFG, German Research Foundation) - Project-ID 317210226 - SFB 1283, and the NSF CAREER award DMS-2339240 and the FIME  research initiative of the Institut
Europlace de Finance.
\bibliographystyle{siam}

\end{document}